\documentclass[12pt]{amsart}
\usepackage{amsmath, amsthm, amssymb, epsf}
\usepackage[dvips]{graphics}

\newtheorem{theorem}{Theorem}
\newtheorem{lemma}[theorem]{Lemma}
\newtheorem{proposition}[theorem]{Proposition}
\newtheorem{corollary}[theorem]{Corollary}
\numberwithin{equation}{section}
\numberwithin{theorem}{section}

\def\({\left(}
\def\){\right)}

\newcommand{\R}{\mathcal{R}}

\newcommand{\A}{\mathcal{A}}
\newcommand{\E}{\mathcal{E}}
\newcommand{\I}{\mathcal{I}}

\newcommand{\ontop}[2]{\genfrac{}{}{0pt}{}{#1}{#2}}

\renewcommand{\S}{\mathcal{S}}

\renewcommand{\a}{\alpha}

\renewcommand{\r}{\rho}
\newcommand{\g}{\gamma}
\renewcommand{\l}{\mathcal{L}}

\renewcommand{\Re}{{\rm{Re}}}
\renewcommand{\Im}{{\rm{Im}}}
\renewcommand{\i}{{\mathrm{i}}}    
\renewcommand{\d}{{\mathrm{d}}}

\begin{document}
\title[Pair correlation of the zeros of $\xi'$]{Pair correlation of the zeros of the derivative of the
Riemann $\xi$-function}

\author{David W. Farmer}
 \address{American Institute of Mathematics, 360 Portage Avenue, Palo Alto,
CA 94306-2244, USA.}
 \email{farmer@aimath.org}

\author{Steven M. Gonek}
 \address{Department of Mathematics, University of Rochester, Rochester, NY 14627, USA}
 \email{gonek@math.rochester.edu}

\thanks{Both authors were supported by
the American Institute of Mathematics and the NSF Focused Research
Group grant DMS 0244660. Work of the second author was also
supported by NSF grants DMS 0201457 and DMS 0653809.}

\begin{abstract}
The complex zeros of the Riemannn zeta-function are identical to the zeros of the
Riemann xi-function, $\xi(s)$.  Thus, if the Riemann Hypothesis is true for the zeta-function,
it is true for $\xi(s)$. Since $\xi(s)$ is entire, the zeros of $\xi'(s)$,  its derivative, would then also
satisfy a Riemann Hypothesis.  We investigate the pair correlation function of the zeros of  
$\xi'(s)$  under the assumption that the Riemann Hypothesis is true. We then deduce 
consequences about the size of gaps between these zeros and the proportion of these  
zeros that are simple.
\end{abstract}

\maketitle

\section{Introduction}

Riemann's xi-function is defined by
\begin{equation}\label{xi}%
 \xi (s) = \frac 12 s (s-1) \pi^{-s/2} \Gamma (s/2) \zeta (s),
\end{equation}
where $\Gamma(s)$ is the Euler $\Gamma$-function and $\zeta(s)$ is
the Riemann $\zeta$-function.
The $\xi$-function
is an entire function of order 1 with functional equation
\begin{equation}\label{fnc equ}%
  \xi (1-s) = \xi (s),
\end{equation}
and  its only zeros are the complex zeros of $\zeta (s)$.  Thus, if 
the Riemann Hypothesis (RH) is true, all the zeros of $\xi(s)$ have real
part one half, and the same would hold for all the zeros of the
derivative $\xi'(s)$.
We assume RH throughout this paper and investigate the
distribution of the zeros  $\r = \frac12 + \i \g$ of $\xi' (s)$.

The distribution of zeros of $\xi'$ is of interest for number-theoretic reasons connected
to the problem of Landau-Siegel zeros, and also in connection to the general behavior
of zeros of entire functions under differentiation.  We discuss those motivations in the
next section.

We calculate  
\begin{equation}\label{F_1}% 
F_1 (\alpha, T) = N_1(T)^{-1}
\sum_{0 < \g, \g' \leq T} T^{\i \alpha (\g -
\g)} w(\g  - \g')\;,
\end{equation}
where the sum is over pairs of ordinates of zeros of $\xi' (s)$ and $w
(u)= 4/(4 + u^2)$ is a weight function.  The normalizing factor
$N_1(T)\sim \frac{1}{2\pi} T\log T$ in front
of the sum is the number of zeros of $\xi'$ with ordinates in $[0,T]$,
which (on RH) differs from the number of zeros of $\xi$ by at most~1.

\begin{theorem}\label{thm:F1} Let $K$ be an arbitrary large positive integer.  
Assuming RH we have
\begin{align*}
F_1 (\alpha, T) =\mathstrut & N(T)^{-1} \sum_{0 < \g, \g' \leq T} T^{\i \alpha
(\g  - \g' )} w (\g  - \g') \\
 =\mathstrut & (1 + o(1))  T^{- 2 |\alpha |}  \log T +  |\alpha | - 4 |\alpha |^2
+ \sum^K_{k=1}
\frac{(k-1)!}{(2k)!} (2 |\alpha |)^{2k+1} + o_K (1)
\end{align*}
%uniformly
as $T\to\infty$, for $|\alpha | < 1$.
\end{theorem}

This theorem is an analogue of Hugh Montgomery's result \cite{M} on the
pair correlation of zeros of the $\zeta$-function.
He considered the function
\begin{equation}\label{eqn:F}%
F (\alpha, T) = N(T)^{-1}
\sum_{0 < \g_0, \g_0' \leq T} T^{\i \alpha (\g_0 -
\g_0)} w(\g_0  - \g_0')\;,
\end{equation}
where the sum is over pairs of ordinates of zeros of $\xi(s)$ and
$w(u)= 4/(4 + u^2)$.  (We use $\g_0$ for zeros
of $\xi$ because $\g$ refers to zeros of $\xi'$ in this paper).
In the terminology of Random Matrix Theory (RMT), $F(\alpha;T)$ is
called the ``2-point form factor'', although sometimes it is mistakenly
referred to as the pair correlation function.  In fact, $F(\alpha;T)$ is
the Fourier transform of the pair correlation function.
Montgomery proved that $F(\alpha;T)$ has main term  $T^{-2 |\alpha|}\log T + |\alpha|$
for $|\alpha|<1$.  That is, $F(\alpha;T)$ behaves like a Dirac $\delta$-function
at~$0$ and is asymptotically $|\alpha|$ when $\varepsilon<|\alpha|<1$.
We see the same $\delta$-like behavior in $F_1(\alpha; T)$; this  is not surprising
since the spike at $\alpha=0$ is simply a consequence of the density of zeros and the weight function.
The behavior  of $F$ and $F_1$ for $0<|\alpha|<1$, however,  is quite different, as illustrated in Figure~\ref{fig:F1}.

\begin{figure}[htp]
\begin{center}
\scalebox{0.75}[0.75]{\includegraphics{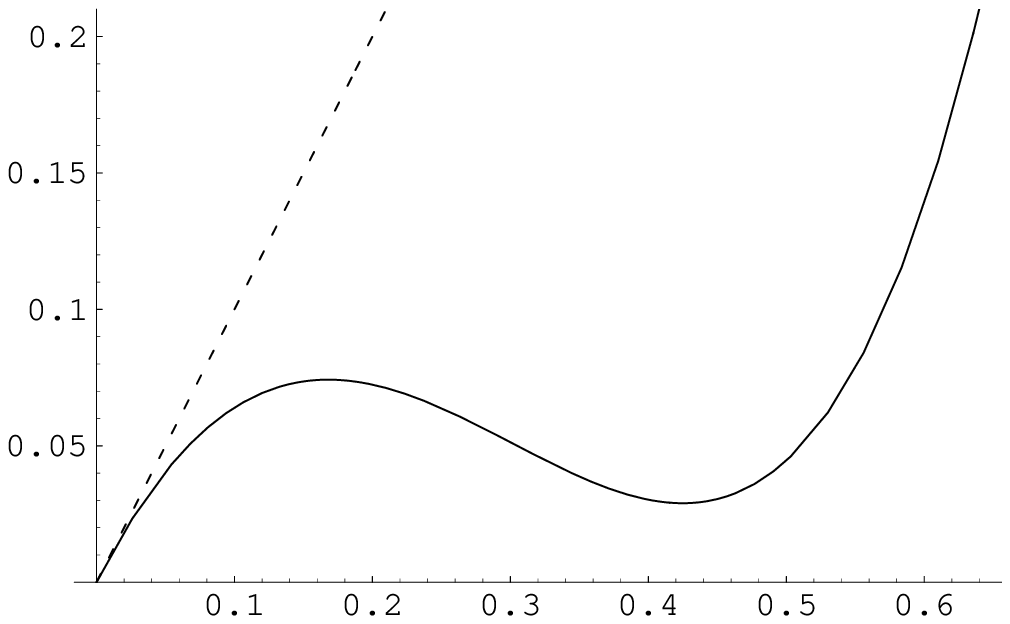}}
\hskip 0.25in
\scalebox{0.75}[0.75]{\includegraphics{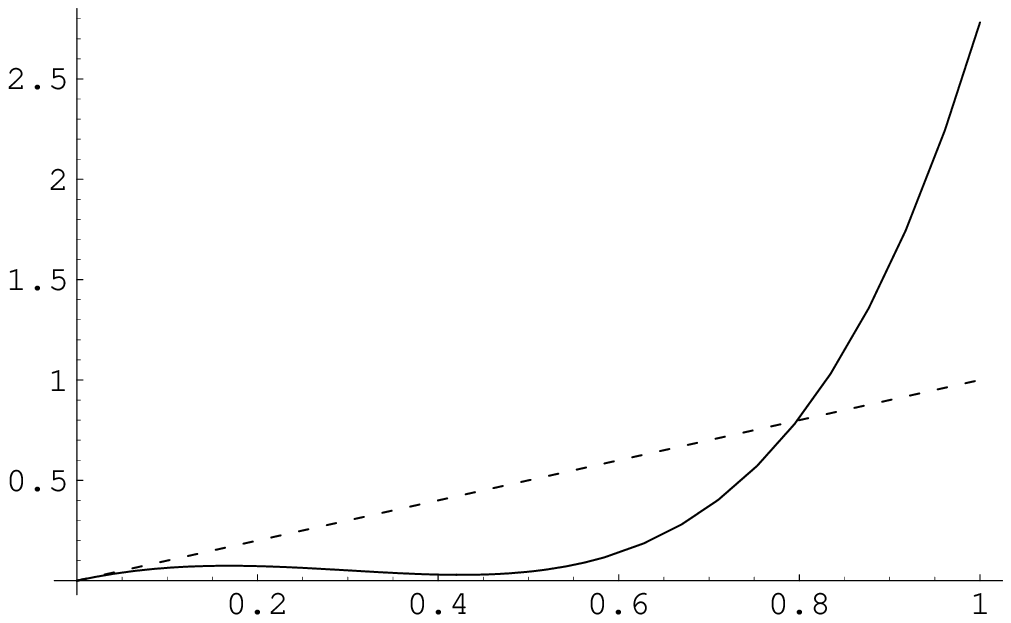}}
\caption{\sf
$F_1(\alpha; T)$ for $0<\alpha<0.64$ and
$0<\alpha<1$.
For comparison, the dotted line in both plots is the form factor $F(\alpha; T)$
for the zeros of the $\xi$-function.
} \label{fig:F1}
\end{center}
\end{figure}

Montgomery used his result on $F(\alpha; T)$ to obtain various estimates on the 
spacing and multiplicity of zeros of the $\xi$-function.
Applying the same methods
to Theorem~\ref{thm:F1} we obtain

\begin{corollary}\label{cor:xipgaps}  Assume RH.
A positive proportion of gaps between zeros of $\xi'$ are
less than 0.91 times the average spacing, and
more than 3.5\% of the normalized neighbor gaps
between zeros of $\xi'$ are smaller than average.
\end{corollary}

\begin{corollary}\label{cor:xipsimple}  Assume RH.
More than 85.84\% of the zeros of $\xi'$ are simple.
\end{corollary}

It is not surprising that the first corollary is weaker than the
corresponding result for the $\xi$-function and the
second is stronger.
The reasons are discussed in the next section.

Conrey~\cite{C} has shown unconditionally that  at least  79.874\% of
the zeros of~$\xi'$ are simple and on the critical line. Thus,
Corollary~\ref{cor:xipsimple} gives a conditional improvement of his result.  
Inserting this estimate into formula~(6) of Farmer~\cite{F}
improves (again on RH) the unconditional estimate there of $(0.63952 + o(1)) N(T)$ 
for the number of  distinct zeros of the $\xi$-function.

\begin{corollary}\label{cor:xidistinct}  Assume RH.
The number of distinct zeros of the zeta-function in $[0,T]$
is larger than $(0.6544 + o(1)) N(T)$.
\end{corollary}

%%%%%%%%%%%%%
%%%%%%%%%%%%%

In the next section we discuss the motivations for our work arising from the
distribution of zeros of entire functions and the problem of Landau-Siegel zeros.
In section~\ref{sec:explicit} we state an explicit formula relating the zeros
of $\xi'$ to  prime numbers, and in section~\ref{sec:startofproof} we begin the
proof of Theorem~\ref{thm:F1} and identify the main terms.
In section~\ref{sec:endofproof} we complete the proof, except for an
arithmetic proposition which is proven in section~\ref{sec:arithmetic}.
In section~\ref{sec:explicitformula} we prove the explicit formula  
used in section~\ref{sec:explicit}. One important investigation we have not carried out 
here is the heuristic determination of  $F_1(\alpha; T)$ when $|\alpha| \geq 1$.  

%%%%%%%%%%%%%
%%%%%%%%%%%%%
%%%%%%%%%%%%%
%%%%%%%%%%%%%

\section{The Alternative Hypothesis and the process of differentiation}\label{sec:motivation}

Montgomery's
study of the statistical behavior of zeros of the Riemann zeta-function 
 was motivated by the problem of Landau-Siegel zeros.  The idea
is that a real zero very close to~$1$ of $L(s,\chi_d)$ would have a
profound effect on the zeros of the Riemann zeta-function:  in a certain
range all the zeros would be on the critical line and would have a 
peculiar spacing. 
Set
\begin{equation}
\tilde{\g}_0=\frac{1}{2\pi}\g_0 \log\left(\frac{\g_0}{2\pi}\right) 
\end{equation}
and denote consecutive zeros of the zeta-function by $\g_0\le\g_0^+$,
so that $\tilde{\g_0}^+ - \tilde{\g_0} $ is 1 on average.
The existence of a Siegel zero implies that in a
certain range almost all the zeros of the zeta-function
satisfy $\tilde{\g}_0^+ - \tilde{\g}_0 > \frac12-\varepsilon$.
Thus, one could
disprove the existence of Landau-Siegel zeros by showing that
$\tilde{\g}_0^+ - \tilde{\g}_0 \le 0.49$, say,  
sufficiently often.
Montgomery's result, however, only allows one to conclude that
$\tilde{\g}_0^+ - \tilde{\g}_0 \le 0.63$
a positive proportion of the time.

Montgomery refers to the connection to Landau-Siegel zeros in his paper~\cite{M},
and similar connections are mentioned in unpublished work of Heath-Brown.  At present
the only published account is due to Conrey and Iwaniec~\cite{CI}.
They show that the existence of Landau-Siegel zeros implies
that, in a certain
range, most of the spacings between consecutive zeros of the zeta function
are close to multiples
of half the average spacing.
  That is, $\tilde{\g}_0^+ - \tilde{\g_0}$ is close to
$\frac12$ or $1$, or $\frac32$, etc.
Although Conrey and Iwaniec give explicit estimates only in the case of
small spacings between zeros, 
their main result 
can be used to show, for example, that if $\tilde{\g}_0^+ - \tilde{\g_0}$
was often close to $0.8$, then there are
no Landau-Siegel zeros.
The estimates in such cases, however,  have not been worked out yet.

If the statistics of the zeros of the zeta-function are governed by 
random matrix theory (RMT),
then one could immediately conclude there are no  Landau-Siegel zeros
because the neighbor spacing is supported on all of $(0,\infty)$.
Since there are no immediate prospects of proving that the zeros of
the zeta-function follow random matrix statistics (or disproving
Landau-Siegel zeros by another method),
it is interesting to probe the boundary of what 
distributions are possible for zeros of the zeta-function.
The following has been proposed as a test case:

\medskip
\noindent\textbf{The Alternative Hypothesis (AH)}
{\sl 
There exists a real number $T_0$ such that if $\g_0>T_0$, then
\begin{equation}
\tilde{\g}_0^+ - \tilde{\g}_0 \in \frac{1}{2} \mathbb Z.
\end{equation}
That is, almost all the normalized neighbor spacings are
an integer or half-integer.
}
\medskip

One can also formulate weaker versions, where the normalized
spacings are approximately integers or half-integers.

\subsection{Consequences of AH}

AH is obviously absurd, but it has not
been disproven.  A~sufficiently strong disproof would show
that there are no Landau-Siegel zeros.
AH implies that Montgomery's function $F(\alpha;T)$ is periodic
with period~two.  Thus, on AH the graph of $F(\alpha;T)$ would look like this:

\begin{figure}[htp]
\begin{center}
\scalebox{1.25}[1.25]{\includegraphics{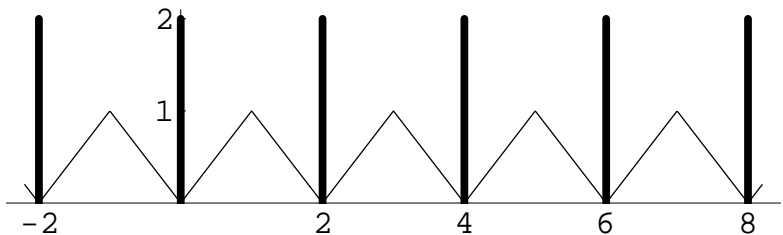}}
\caption{\sf
$F(\alpha;T)$ on the Alternative Hypothesis (the heavy vertical
lines represent the Dirac $\delta$-functions at the even integers).
} \label{fig:AH}
\end{center}
\end{figure}

Similarly, the triple correlation function form factor  $F(\alpha,\beta;T)$ 
(see \cite{H,RS,FG}) would be periodic with period two in both
$\alpha$ and $\beta$.  Unfortunately, this does not seem to completely
determine $F(\alpha,\beta;T)$ because $F(\alpha,\beta;T)$ is only
known on the hexagon $|\alpha|<1$, $|\beta|<1$, $|\alpha - \beta |<1$.

If we let $g_\mu$ denote the proportion of $\tilde{\g}_0$ such that
$\tilde{\g}_0^+ - \tilde{\g}_0=\mu$, a straightforward calculation
assuming AH leads to  $g_0=0$, $g_{\frac12}=\frac12 - \frac{2}{\pi^2} \approx 0.297$,
and $0.405\approx \frac{4}{\pi^2} \le g_1 \le \frac12 $.
Thus, one could disprove AH by showing that more than 30\% of
the normalized neighbor gaps of $\xi$ are less than $0.999$.
We note that RMT predicts that 53.39\% of the neighbor gaps are smaller
than average.  Montgomery's result implies that more than 12.3\%
of the neighbor gaps are smaller than average (set $\lambda=1-\varepsilon$ at
the bottom of page 192 of~\cite{M}).   Corollary~\ref{cor:xipgaps} gives
information about consecutive small gaps between zeros of~$\xi$.
This is discussed further in the next subsection.

One wonders whether AH and the existing results on
zero correlations determine the distribution of neighbor spacings.
To specify $g_1$ seems to require using the triple correlation
$F(\alpha,\beta;T)$ to determine how often  two consecutive
normalized neighbor gaps of size~$\frac12$ can occur.  So this question may be
equivalent to the question of whether AH determines all the
correlation functions.  

%Goldston's conjecture states
%$\int_1^\infty F(\alpha;T)\alpha^{-2} d\alpha \sim 1$,
%which contradicts the AH value of $0.9596$.

Goldston and Montgomery \cite{GM} showed that the pair correlation conjecture
is equivalent to
\begin{equation}\label{eqn:primevar}
\int_1^{X} \left(\psi(x+h)-\psi(x)-h \right)^2
\,dx  \sim h X \log\left(\frac{X}{h}\right),
\end{equation}
for $h$ in a certain range depending on~$X$.
Here  $\psi(x)=\sum_{n\le x}\Lambda(n)$, where
$\Lambda$ is the von~Mangoldt function:
$\Lambda(n)=\log p$ if $n=p^m$ with $p$ prime, and $0$ otherwise.
Montgomery and Soundararajan~\cite{MS} interpret \eqref{eqn:primevar}
as saying 
$\psi(x+h)-\psi(x)$ has mean $h$ and variance $h \log\left(\frac{X}{h}\right)$,
and they note that the Cram\'er model of the primes predicts a larger
variance of $h \log{X}$.
It would be interesting to see what the right-hand  side of 
\eqref{eqn:primevar} equals if one assumes the Alternative Hypothesis.
   
These connections
indicate the value of studying the statistics of the zeros of the
zeta-function.  We now explain the connection to the zeros of~$\xi'$.

%%%%%%%%%%%%%
%%%%%%%%%%%%%

\subsection{Zeros of derivatives}  The statistics of the zeros of
$\xi'$ are interesting because of their connection with the Alternative Hypothesis 
and also as an illustration of the  general behavior of  the zeros of  
derivatives of an entire function upon repeated differentiation.

One motivation for studying the analogue of Montgomery's
function for the zeros of $\xi'$ is the expectation
that our Theorem~\ref{thm:F1} might contradict~AH.  The zeros of
$\xi'$ are influenced by the zeros of $\xi$ in complicated ways,
so it seems unlikely that RMT and AH would predict the same behavior
for $F_1(\alpha;T)$ for $|\alpha|<1$.  But there are several caveats.
First, as described in the previous section, it is not known whether
or not AH determines all the correlation functions of the zeros.
This may lead to some flexibility in  $F_1(\alpha;T)$ for $|\alpha|<1$,
which may be consistent with Theorem~\ref{thm:F1}.
Second, it is not known how to transfer a measure on the zeros of
$\xi$ to a measure on the zeros of~$\xi'$.  Thus, even if AH determined all
the correlations of the zeros of~$\xi$, it is still an unsolved
problem to determine the correlations of the zeros of~$\xi'$.
Third, merely contradicting AH is not sufficient to disprove
the existence of Landau-Siegel zeros.
AH~is an extreme example of a possible consequence of Landau-Siegel zeros.
Presumably an extension of the work of Conrey and Iwaniec~\cite{CI}
would show that Landau-Siegel zeros imply that $F(\alpha;T)$ approximately
follows Figure~\ref{fig:AH} for some range of~$\alpha$.

Another motivation is to understand the general behavior of zeros under
differentiation.  The Riemann $\Xi$-function is defined as
$\Xi(z)=\xi(\frac12+i z)$.  The $\Xi$-function is an entire function
of order~1 that is real on the real axis.
For such functions, repeated differentiation causes the zeros
to migrate to the real axis~\cite{CCS, Kim, KKi}.  Thus, in any bounded region
the Riemann Hypothesis is true for the $n$th derivative
$\Xi^{(n)}(z)$ for sufficiently large~$n$.

It is conjectured~\cite{FR} that for real entire functions of order~1,
whose zeros lie in a strip around the real axis,
not only do the zeros migrate toward the real axis, but they also
approach equal spacing.  That is, the derivatives approach  a multiple of $e^{ax} 
 \cos(bx+c)$.  This conjecture
has been proven with some restrictions on the distribution
of zeros~\cite{FR} and for some special cases, such as the
$\Xi$-function~\cite{Ki} and the reciprocal of the gamma function~\cite{B}.

The reason differentiation leads to equally spaced zeros is that,
locally, the zeros of $f'$ move away from concentrations of zeros
of $f$ and towards regions with fewer zeros of~$f$.  Thus, small
gaps become larger and large gaps become smaller. Figure~\ref{fig:xiplot}
illustrates these ideas.

\begin{figure}[htp]
\begin{center}
\scalebox{1.3}[1.3]{\includegraphics{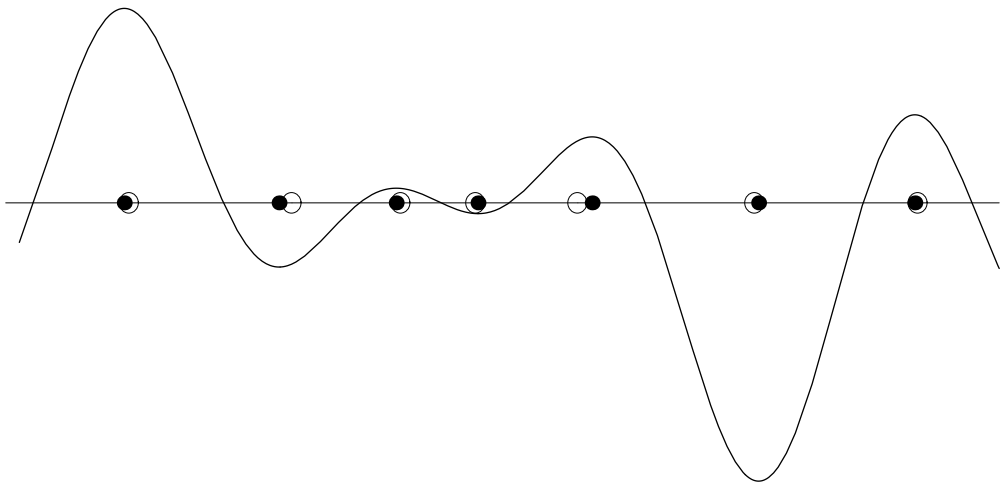}}
\caption{\sf The circles are the midpoints of neighboring zeros of~$f$,
and the dots are the zeros of~$f'$.} \label{fig:xiplot}
\end{center}
\end{figure}

A consequence is that if $\xi'$ has a small gap between consecutive zeros,
then $\xi$ must have had an even smaller gap.  Thus, one should expect that the
existence of Landau-Siegel zeros could be disproven by showing that
$\xi'$ has sufficiently many zeros separated by less than half
the average spacing.  That is, one should be able to extend the result of
Conrey-Iwaniec~\cite{CI} to the zeros of $\xi^{(n)}$ for any $n\ge 1$.
However, we have not worked out the necessary details to show that the
required number of small gaps between zeros of $\xi'$ leads to the
same number of small gaps between the zeros of~$\xi$. Thus, at present
we just mention this as a motivating principle.  See \cite{FR} for a
discussion.

Based on these ideas, we see that Corollary~\ref{cor:xipgaps} does
not contradict the Alternative Hypothesis.  On average each normalized
gap of size $\frac12$ is adjacent to a gap of size $\frac12$ or~$1$,
so AH implies that at least 29.7\% of the normalized neighbor gaps of $\xi'$ are
smaller than~1.  If Corollary~\ref{cor:xipgaps} could be improved to
show that $\xi'$ has normalized zero gaps smaller than~$0.75$, then that
would imply, on AH, that $\xi$ has consecutive gaps of size~$\frac12$.

Our final motivation is to understand the manner in which differentiation
causes the zeros to become equally spaced.  Let $F_n$ denote the
analogue of $F_1$  involving the zeros of~$\xi^{(n)}(z)$ 
rather than the zeros of ~$\xi'(z)$.  
Since the zeros of $\xi^{(n)}(z)$   approach 
equal spacing as $n$ increases,  $F_n$ approaches a sum of Dirac $\delta$-functions
supported at the integers.  We would like to understand this transition
to a sum of $\delta$-functions.
The only case we know of where this has been worked out explicitly
is for random trigonometric polynomials~\cite{FY}.  Theorem~\ref{thm:F1}
is a first step in this direction, and Figure~\ref{fig:F1} does seem
to illustrate the expected behavior.
In a recent Ph.D. thesis, Jim Bian has now
worked out explicit formulas for $F_n(\alpha;T)$ for $n\ge 2$ and
$|\alpha|<1$.

\subsection{Connection to RMT}
It would be interesting to know the random matrix analogue of Theorem~\ref{thm:F1}.
At present, this is one of the few calculations   carried out for the Riemann 
zeta-function for which a  random matrix analogue seems out of reach.
One issue is that in the random matrix world there seems to be no direct
analogue of the Riemann $\xi$-function. See Section 1.2 of \cite{CFKRS}
for a discussion. The closest match is
\begin{equation}\label{eqn:lambda}
z^{-n/2} \Lambda(z),
\end{equation}
where $\Lambda(z)$ is the characteristic polynomial of a random matrix
from the unitary group~$U(n)$, chosen uniformly with respect
to Haar measure.
However, this is more  nearly  an analogue of the Hardy
$Z$-function  defined by $Z(t)=\chi(\frac12+i t)^{-\frac12}\zeta(\frac12+i t)$.
This was the motivation of Conrey and Soundararajan~\cite{CS}, who did
similar calculations to ours for the $Z$-function.
But it turns out that the zeros of $\Xi'(t)$ and $Z'(t)$
should have similar correlation functions.
Since $\Xi(t)$ is approximately $e^{-\frac{\pi t}{2}} Z(t)$, the
corresponding zeros of $\Xi'(t)$ and $Z'(t)$ generally differ
by $O(1/\log^2 t)$.  So one should expect the correlation functions of
their zeros to be equal, to leading order.  We provide a more rigorous
explanation in Section~\ref{sec:Z}.  And those correlation
functions should equal, to leading order,
the correlation
function of the zeros of the derivative of
the completed characteristic polynomial~\eqref{eqn:lambda}.

%%%%%%%%%%%%%
%%%%%%%%%%%%%
%%%%%%%%%%%%%
%%%%%%%%%%%%%

\section{An explicit formula and first steps of the proof}\label{sec:explicit}

The general outline of our calculation is similar to the case
of zeros of the zeta-function.  We start with an explicit formula
relating a sum over zeros of $\xi'$ to a sum involving number-theoretic
expressions, and then compute the mean-square of both sides.
Our main complication is that the number-theoretic side of
the explicit formula is not a Dirichlet series, but
an ``approximate'' Dirichlet series, by which
we mean a Dirichlet series whose coefficients may depend on~$s$.

The explicit formula is derived in Section~\ref{sec:explicitformula}.
Here we just outline the calculation and state the formula.

We begin by noting that $\xi'' / \xi' (s)$ has simple poles  with
residue~1  at the zeros of $\xi' (s)$ and no others; here  
zeros are counted as many times as their multiplicity.  Our first goal is to write
$\xi'' / \xi' (s)$ as an ``approximate'' Dirichlet series.
From the definition of the $\xi$-function \eqref{xi} we have
\begin{equation}\label{xi'/xi}%
\frac{\xi'}\xi (s) =  L(s) + \frac{\zeta'}\zeta (s)  \;,
\end{equation}
where 
\begin{align}\label{L defn} %
L (s) =\mathstrut & \frac 1s + \frac 1{s-1} - \frac{\log \pi}2 +
\frac 12 \frac{\Gamma'}{\Gamma} \left(\frac{s}{2}\right)\cr
=\mathstrut & \frac 12 \log \frac s{2\pi} + O \left( \frac1{|s| +2} \right)
\end{align}
and
\begin{equation}\label{L ineq 2}%  
 L' (s) \ll \frac 1{|s| +2}
\end{equation}
in the region $|\arg s| < \pi - \delta$, $|s| \geq 1/4$, say.

Multiplying both sides of \eqref{xi'/xi} by $\xi (s)$ and calculating the
logarithmic derivative leads to the following lemma, the proof
of which is in Section~\ref{sec:explicitformula}.
\begin{lemma}\label{lem:logderiv} For $\sigma \geq 1 + \varepsilon$, $|t| \geq T_\varepsilon$,
and $K$ a large positive integer
we have
\begin{equation}\label{Dir series}%
 \frac{\xi''}{\xi'} (s) = L(s) + \sum^\infty_{n=1}
 \frac{a_K(n,s)}{n^s} + O\left(\frac{1}{\varepsilon^2 2^K}\right)  \;.
\end{equation}
Here we have written
\begin{equation}
a_K (n,s) = \sum^{K}_{k=0}\frac{\a_k(n)}{L(s)^{k}},
\end{equation}
where
\begin{equation}\label{eqn:alphak}
\a_j(n) =
\begin{cases}
 -\Lambda(n) & \quad \hbox{if} \quad  j=0,\cr
\Lambda_{j-1} * \Lambda \log(n) &  \quad\hbox{if}  \quad j\ge 1 .
\end{cases}
\end{equation}
\end{lemma}

The function $\Lambda_j$ for $j\ge 0$ is the $j$-fold convolution of
the von~Mangoldt function, defined by
\begin{equation}\label{eqn:vonM}
 \left( - \frac{\zeta'}\zeta (s) \right)^j = \sum^\infty_{n=1}
\frac{\Lambda_j (n)}{n^s} \,.
\end{equation} 
for $\sigma >1$.  Simple estimates of $\Lambda_j$ and $\alpha_k$ will be
used repeatedly.  We have the trivial bound
\begin{equation}\label{Lambda_j bound}
\Lambda_j (n) \leq (\log n)^j \qquad (j = 1,2, \dots) \,,
\end{equation}
which follows from iterating
\begin{equation}
\Lambda_j (n) = \sum_{d|n} \Lambda (d) \Lambda_{j-1} \left( \frac nd \right) \leq
\left( \max_{d|n} \Lambda_{j-1} \left( \frac nd \right) \right)
\sum_{d|n} \Lambda (d) \leq \Lambda_{j-1} (n^*)\log n 
\end{equation}
for some divisor $n^*$ of $n$. Similarly,
\begin{equation}\label{Alpha_k bound}
\a_k(n)= \Lambda_{k-1} * \Lambda \log(n)
 \leq (\log n)^{k+1} \qquad (k = 1,2, \dots)\,,
\end{equation}
from which it follows that
\begin{equation}\label{a_K (n,s) bd}%
|a_K (n,s)| \leq \log n  \sum^K_{k=0} \left( \frac{\log n}{|L(s)|}
 \right)^{k} \ll \log^{K+1}n
\end{equation}
when $-1 \leq \sigma \leq 2$, say.

A contour integral of $\xi''/\xi'$ times an appropriate kernel leads to the
following explicit formula; the details are in Section~\ref{sec:explicitformula}.

\begin{proposition}\label{prop:explicitformula} For
$5/4<\sigma< 2$, $x\geq 1$, $0< \varepsilon <1/8$, and $K$ a positive integer
\begin{equation}
\begin{aligned}
(2 \sigma - 1)\sum_{\gamma} &
\frac{x^{\i \gamma}}{(\sigma-1/2)^2 + (t- \gamma)^2} \\
= &  x^{-1/2} \left(
\sum_{n\leq x} a_K (n, 1-\bar s ) \left( \frac xn \right)^{1-\bar s} +
\sum_{n>x} a_K (n,s) \left( \frac xn \right)^s \right) \\
&   +   x^{1/2 - \bar s}  \log \frac\tau{2\pi} + O (  x^{1/2 - \sigma} )
+ O_{\varepsilon, K} \left(
    x^{1/2} \tau^{-1} \max (x^\varepsilon, \log^{2K+2} x)
\right) \,,
\end{aligned}
\end{equation}
where $\tau=|t|+2$.
\end{proposition}

%%%%%%%%%%%%%
%%%%%%%%%%%%%
%%%%%%%%%%%%%
%%%%%%%%%%%%%

\section{Beginning of the Proof}\label{sec:startofproof}

Set $\sigma = 3/2$ in Proposition~\ref{prop:explicitformula}, write the resulting 
equation as $L(x, t) = R_1(x, t) +R_2(x, t) + R_3(x, t)+ R_4(x, t)$, and calculate
\begin{equation}\label{L=R}
\int_0^T |L(x, t)|^2 \d t = \int_0^T | R_1(x, t) +R_2(x, t) + R_3(x, t)+ R_4(x, t) |^2 \d t\,.
\end{equation}
 
The left-hand side may be treated in exactly the same way 
as the corresponding expression in Montgomery~\cite{M}, 
to whom we refer the reader (cf. pp. 187--188). 
We find
\begin{equation}\label{L}
\int_0^T |L(x, t)|^2 \d t 
= 2\pi \sum_{0 < \g, \g' \leq T} x^{\i (\g -
\g')} w (\g - \g') + O (\log^3 T)\;,
\end{equation}
where $w (u) = 4/(4 + u^2)$.

Next we begin the calculation of the right-hand side of \eqref{L=R}. In various ranges of 
$x$, one or another of the integrals    $\int_0^T |R_i(x, t)|^2 \d t $ dominates the others,
and we record the following useful formula for later. Given $x$,
let $ {\bf \mathcal{R}_1}(x)$ be the largest of $\int_0^T |R_i(x, t)|^2 \d t  \; (i=1, 2, 3, 4)$,
 $ {\bf \mathcal{R}_2}(x)$   the next largest, and  so on.   Then we have
\begin{equation} \label{int of R inequality}
\int_0^T |R_1(x, t) +R_2(x, t) + R_3(x, t)+ R_4(x, t) |^2 \d t 
=  {\bf \mathcal{R}_1}(x) + O\left( \left( {\bf \mathcal{R}_1}(x) {\bf \mathcal{R}_2} (x) \right)^{1/2} \right)
\end{equation}

We find
\begin{equation}\label{R_2}
\int_0^T |R_2 (x, t)|^2 \d t 
= \int^T_0 \left| x^{-1+ \i t} \log \frac{\tau}{2\pi}  \right|^2 \d t =
\frac{T}{ x^2} \left( \log^2 T + O (\log T) \right) \,,
\end{equation}
\begin{equation}\label{R_3}
\int_0^T |R_3(x, t)|^2 \d t \ll   \frac{T}{x^2} \,
\end{equation}
and
\begin{equation}\label{R_4}
\begin{aligned}
\int_0^T |R_4(x, t)|^2 \d t 
& \ll_{\varepsilon, K}  x \max (x^{2\varepsilon}, (\log x)^{4K+4})
\int_0^T  \frac{\d t}{\tau^2} \\
& \ll_{\varepsilon, K} x \max ( x^{2\varepsilon}, (\log x)^{4K + 4})\;.
\end{aligned}
\end{equation}

To estimate $\int_0^T |R_1(x, t)|^2 \d t $, we need the following 
\begin{lemma}\label{Integral lemma} 
Let $\sigma = -1/2$ or $3/2$. Then for $ x >0 $ we have
\begin{align}\label{integral 1}
\int^T_0
\frac{x^{\i t}}{\overline{L(s)}\,^k L(s)^\ell} \; \d t 
=
\begin{cases} 
 T \left(\frac12 \log \left(\frac{T}{2\pi} \right)\right)^{-(k+l)}   
 \left( 1 + O_K\left( \frac{1}{\log T} \right) \right)
&\text{if}\; x=1 \,, \\
O_K\left(\frac{1}{|\log x |} \right)
&\text{if}\; x \neq 1 \,,
\end{cases}
\end{align}
where $k, \ell = 0,1,2,\dotsc,  K$. The same result holds for
\begin{equation}\label{integral 2}
\int^T_0
\frac{x^{\i t}}{\overline{L(1-\bar s)}\,^k L(s)^\ell} \; \d t \;.
\end{equation}
\end{lemma}

\begin{proof} We prove \eqref{integral 1} only, as the proof of \eqref{integral 2} 
is almost identical.

First observe that from \eqref{L defn} we have 
\begin{equation}\label{log formula}
\frac{1}{\overline{L(s)}\,^k L(s)^\ell} = 
\frac{1}{\left(\frac12 \log \tau/2\pi \right)^{k+l} }
\left(1 + O_K\left( \frac{1}{\log \tau} \right)   \right) \;.
\end{equation}
The case $x=1$ follows immediately. 
 
Now suppose that  $x>0$, but $x\neq 1$.
Integrating by parts, we find that our integral equals
$$  
\frac{x^{\i t}}{\i \log x \,\overline{L(s)}\,^k
L(s)^\ell} \bigg|^T_0 + \frac 1{\i \log x} \int^T_0
\frac{x^{\i t}}{ \overline{L(s)}\,^k L(s)^\ell } 
\left( k \frac{ \overline{L'}(s)}{\overline{ L }(s)} + \ell \frac{L'(s)}{L(s)} \right) \d t\;.
$$
By \eqref{L defn}, \eqref{L ineq 2}, and \eqref{log formula}, this is
\begin{align*}
&\ll  \frac 1{|\log x|} \left( 1 + (k + \ell)  \int_0^T \;
\frac{1}{ \tau(1/2 \log \tau/2\pi)^{k + \ell +1}}\d t   \right) \\
&\ll_K   \frac 1{| \log x|} \,.
\end{align*}
This completes the proof.

\end{proof}

We now come to the term
\begin{align*}
\int_0^T |R_1(x, t)|^2 \d t  & = \frac 1x \int^T_0 \left| \sum_{n\leq x} a_K (n, -1/2 + \i t) \left(
\frac xn \right) ^{-1/2 + \i t} + \sum_{n>x} a_K (n, 3/2 + \i t) \left(
\frac xn \right)^{3/2 + \i t} \right|^2 \d t \\
& = \frac 1{x^2} \int^T_0 \left| \sum_{n\leq x} a_K (n, - 1/2 + \i t)
n^{1/2 - \i t} \right|^2  \d t 
+ x^2 \int^T_0 \left| \sum_{n>x} a_K (n, 3/2 + \i t)n^{-3/2 - \i t}
\right|^2 \d t \\
&\quad  + 2 Re\, \int^T_0 \overline{\left( \sum_{n\leq x} a_K (n, - 1/2 +
\i t) n^{1/2 - \i t}\right)} \left( \sum_{m>x} a_K ( m,  3/2 + \i t) m^{-
3/2 + \i t}\right) \d t\\
& = \frac 1{x^2} R_{1,1}  + x^2 R_{1,2}  + 2 \,\Re\, R_{1,3}  \;,  
\end{align*}
say. Recalling that
$  
a_K (n,s) = \sum^K_{k=0} \a_k (n)/L(s)^k,
$ 
we see by Lemma~\ref{Integral lemma}  that 
\begin{align*}
 R_{1,1} & = \sum_{m,n\leq x} \sqrt{mn} \sum^K_{k,\ell =0} \overline{\a_k (n)}
\a_\ell (m) \int^T_0 \frac{(n/m)^{it}}{\overline{L(-1/2 + it)^k} L (-1/2
+ it)^\ell} \; dt\\
& = (1 + o (1)) T \sum_{n\leq x} n \left( \sum^K_{k,\ell=0}
\frac{\overline{\a_k(n)} \a_\ell (n)}{(1/2 \log T /2\pi)^{k+\ell}}
\right)\\
& \qquad + O_K \left( \sum_{1\leq m < n \leq x} \sqrt{mn} \left(
\sum^K_{k,\ell =0}  \frac{|\a_k (n) \a_\ell (m)|}{\log n/m }
  \right) \right)\\
& = (1+ o(1)) T \sum_{n\leq x} n \left| \sum^K_{k=0} \frac{\a_k (n)}{(1/2
\log T/2\pi)^k} \right|^2 \\
& \qquad + O_K \left( \sum_{1 \leq m < n \leq x} \frac{\sqrt{mn}}{\log n/m}
  \left( \sum^K_{k=0} |\a_k (n)|\right)
\left( \sum^K_{k=0} |\a_k(m)|\right) \right)  \;.
\end{align*}
From the standard inequality
\begin{equation}
 \sum_{1 \leq m < n \leq x} \frac{|b(m) b(n)|}{(mn)^\sigma 
\log n/m} \ll x \log x \sum_{n\leq x} \, \frac{|b(n)|^2}{n^{2\sigma}} 
\end{equation}
and \eqref{Alpha_k bound} we see that the $O$-term is  
\begin{align*}
& \ll_K x \log x \sum_{n\leq x} n \left( \sum^K_{k=0} |\a_k (n)|\right)^2
\\
& \ll_K x \log x \sum_{n\leq x} n (\log n)^{2K+2} \\
& \ll_K x^3 (\log x)^{2K+3}\;.
\end{align*}
Thus, we have
$$ 
 R_{1,1} = (1 + o(1)) T \sum_{n\leq x} n \left| \sum^K_{k=0} \frac{\a_k
(n)}{(1/2 \log T/2\pi)^k} \right|^2 + O_K \left(x^3 (\log x)^{2K+3}\right )\;.
$$

Similarly, we find that
$$
 R_{1,2}  = (1 + o (1))T \sum_{n>x} n^{-3} \left| \sum^K_{k=0} \frac{\a_k
(n)}{(1/2 \log T/2\pi)^k} \right|^2 + O_K  \left(x^{-1}(\log
x)^{2K+3} \right)\;.
$$
  
For $ R_{1, 3} $ we have  
\allowdisplaybreaks{\begin{align*}
 R_{1, 3}  & = \sum_{n\leq x} n^{1/2} \sum_{m>x} m^{-3/2} \sum^K_{k,\ell =0}
\overline{\a_k (n)} \a_\ell (m) \int^T_0 \frac{ (mn)^{it}}{\overline{L (-
1/2 + it)}\, ^k L (3/2 + it)^\ell} \, dt\\
& \ll \sum_{n\leq x} n^{1/2} \sum_{m>x} m^{-3/2} (\log mn)^{-1} \left(
\sum^K_{k,\ell =0} |\a_k (n) \a_\ell (m) |   \right) \\
& \ll_K \sum_{n\leq x} \sum_{m>x} \frac{n^{1/2} m^{-3/2}}{\log mn} (\log
m)^{K+1} (\log n)^{K+1} \\
& \ll_K \frac 1{\log x} \left( \sum_{n \leq x} n^{1/2} (\log n)^{K+1}
\right) \left( \sum_{m>x} m^{-3/2} (\log m)^{K+1} \right) \\
& \ll_K x (\log x)^{2K + 1} \;.
\end{align*} }

Combining these estimates, we have
\begin{equation} 
\begin{aligned}\label{R1}
\int_0^T |R_1(x, t)|^2 \d t    = &(1 + o(1)) T x^{-2} \sum_{n\leq x} n \left| \sum^K_{k=0} \;
\frac{\a_k (n)}{(1/2 \log T/2\pi)^k} \right|^2 \\
 &\; + (1 + o(1)) Tx^2 \sum_{n>x} \frac 1{n^3} \left| \sum^K_{k=0}
\frac{\a_k(n)}{(1/2 \log T/2\pi)^k} \right|^2 \\
 & \quad + O_K (x (\log x)^{2K+3})\;.
\end{aligned}
\end{equation}

In the next section we complete the proof of Theorem~\ref{thm:F1},
subject to an arithmetic proposition which we prove in Section~\ref{sec:arithmetic}.
 
 %%%%%%%%%%%
 %%%%%%%%%%%
 %%%%%%%%%%%
 %%%%%%%%%%%

\section{Proof of Theorem~\ref{thm:F1}}\label{sec:endofproof}

It remains to evaluate~\eqref{R1} and to put the expressions in the
form of Theorem~\ref{thm:F1}.
The main terms from~\eqref{R1} can be obtained from a Stieltjes integral
involving
\begin{equation}\label{eqn:Ax}
A(x) = A(x,K,T) = \sum_{n\leq x} \left| \sum^K_{k=0}
\frac{\a_k (n)}{\l^k} \right|^2 .
\end{equation}
Here $\alpha_k$ is defined in \eqref{eqn:alphak}
and $\l = \frac 12 \log \frac{T}{2\pi}$.
We write
\begin{equation}\label{A(x) 3}
A(x) = \sum^K_{k=0} A_{k,k} (x) \l^{-2k} +
2 \sum^K_{k=1} \sum_{0 \leq \ell <k} A_{k,\ell} (x) \l^{- (k+\ell)},
\end{equation}
where
\begin{equation}\label{A(x) 2}
A_{k,\ell} (x)=
\sum_{n\leq x} \a_k (n) \a_\ell (n).
\end{equation}
The following proposition, which is proven in Section~\ref{sec:arithmetic},
is sufficient to evaluate the leading order
asymptotics of~\eqref{R1}.

\begin{proposition}\label{prop:Akell} We have \begin{equation}
A_{k,0} (x) =
\begin{cases}
x \log x + O (x)  \quad& \text{ if }\ \  k=0 \,,\\
- x \log^2 x + O(x \log x) \quad& \text{ if }\ \  k=1 \,,\\ 
O\bigl(x^{\tfrac12+\varepsilon}\bigr)
\quad
&\text{ if }\ \  k \geq 2 \;.
\end{cases}
\end{equation}
and
\begin{equation}
A_{k,1} (x) =
\begin{cases}
 x \log^3 x + O (x \log^2 x) \quad &
\text{ if }\ \ k=1\;,\\
O\bigl(x^{\tfrac12+\varepsilon}\bigr)
&
\text{ if }\ \ k \geq 2
\;. \end{cases}
\end{equation}
If $k\ge 1$, then
\begin{equation}
A_{k,k}(x) = 2 \frac{(k-1)!}{(2k)!} x\log^{2k+1} x + O(x \log^{2k} x) ,
\end{equation}
and if
$k>\ell\ge 2$, 
then
\begin{equation}\label{eqn:Aklbnd}
A_{k,\ell}(x)\ll x \log^{k+\ell} x .
\end{equation}
As a consequence, we have
\begin{equation}\label{The Sum Fnc}
{A}(x) = x \log x \left( 1 -2 \left( \frac{\log x}\l \right) + 2
\sum^K_{k=1} \frac{(k-1)!}{(2k)!} \left( \frac{\log x}\l \right)^{2k}
\right) + O_K (x)\end{equation}
for $x = T^\alpha$ with $0 <\alpha \le C_0$, where $C_0$ is any fixed positive number.
\end{proposition}

\begin{proof}[Proof of Theorem~\ref{thm:F1}]

We first evaluate $\int_0^T |R_1(x, t)|^2 \d t $ 
given by \eqref{R1}.  We have
\begin{align}\label{M formula}
\int_0^T |R_1(x, t)|^2 \d t   & =
(1 + o(1)) T \left\{ x^{-2} \int^x_{1^{-}} u \, d {A} (u)
+ x^2 \int^\infty_x u^{-3} d {A} (u) \right\} 
+ O_K \left(x (\log x)^{2K+3} \right)\\
& = (1 + o (1) T \left\{ - x^{-2} \int^x_1 {A} (u) du + 3x^2
\int^\infty_x u^{-4} {A} (u) du \right\}  
+ O_K \left(x (\log
x)^{2K+3} \right) \;.
\end{align}
Note that the boundary terms canceled when we integrated by parts.

The typical term in the series for ${A} (u)$ has the form 
 $C u  (\log u)^{m+1}/ \l^m$ and, assuming that $m \ll K$, 
we have
$$ 
\frac 1{\l^m} \int^x_1 u (\log u)^{m+1} du = \frac 1{\l^m} \frac{x^2}2
(\log x)^{m+1} \left( 1 + O_K \left( \frac{1}{\log x} \right)
\right) 
$$
and
$$ \frac 1{\l^m} \int^\infty_x u^{-3} (\log u)^{m+1} du = \frac 1{\l^m}
\frac 1{2x^2} (\log x)^{m+1} \left( 1 + O_K \left( \frac{ 1}{\log x}
\right) \right)\;.
$$
Using the first formula  and \eqref{The Sum Fnc}, we have
\begin{align*}
  - x^{-2} \int^x_1 {A} (u) du 
  = &-x^{-2} \int^x_1 u \left( \log u - 2 \frac{\log^2 u} \l 
  + 2 \sum^K_{k=1} \frac{(k-1)!}{(2k)!} 
\frac{ (\log u)^{2k+1} }{\l^{2k}} \right) du + O_K (1) \\
= & \log x \left\{ -\frac 12 + \frac{\log x}\l - 
\sum^K_{k=1} \frac{(k-1)!}{(2k)!} \left( \frac{\log x}{\l} \right)^{2k} \right\} 
\left( 1 + O_K \left( \frac{1}{\log x} \right) \right)
 + O_K (1) \,.
\end{align*}
For $x\ll T^C$,  the $O_K(1/ \log x)$ term 
contributes no more than $O_K(1)$ to this.  Hence,
\begin{equation}\label{int 1}
 -x^{-2} \int^x_1 {A} (u) du = \log x \left( - \frac 12 +
\frac{\log x}\l - \sum^K_{k=1} \frac{(k-1)!}{(2k)!} \left( \frac{\log
x} \l\right)^{2k} \right) + O_K (1) \;.
\end{equation}
Similarly, using the second formula  and \eqref{The Sum Fnc}, we obtain  
\begin{align}\label{int 2}
 3x^2 \int^\infty_x u^{-4} {A} (u) du & = 3x^2
\int^\infty_x u^{-3} \left( \log u - \frac{2\log^2 u} \l + 2 \sum^K_{k=1}
\frac{(k-1)!}{(2k)!} \frac{(\log u)^{2k+1}}{ \l^{2k}} \right) du + O_K (1)  
\notag   \\
& =  \log x \left\{ \frac 32  - 3 \frac{\log x}L +  3 \sum^K_{k=1} \frac{(k-
1)!}{(2k)!} \left( \frac{\log x}L \right)^{2k} \right\}  + O_K(1) \,.
\end{align}
Combining terms we obtain
\begin{equation}\label{ } 
\begin{aligned}
\int_0^T |R_1(x, t)|^2 \d t   = &\left(1 + o (1) \right) T \log x \left( 1 -2 \frac{\log x}{\l}
+ 2
\sum^K_{k=1} \frac{(k-1)!}{(2k)!} 
\left( \frac{\log x} \l \right)^{2k} \right) \\
& + O_K (T) + O_K \left(x (\log x)^{2K +3} \right)
\end{aligned}                  
\end{equation}
provided $x\ll T^{C_0}$. 

Recall from  \eqref{L=R} and \eqref{L} that we write
\begin{equation}\label{Sum over gammas}
2\pi  \sum_{0 < \g, \g' \leq T} x^{\i (\g - \g')} w (\g - \g') 
=   \mathcal{R}_1(x) + O\left( \left(\mathcal{R}_1(x)\mathcal{R}_2 (x) \right)^{1/2} \right)
  +    O (\log^3 T)\,,
\end{equation}
where, for a given $x$,  $\mathcal{R}_1(x)$  is the largest of  $\int_0^T |R_i(x, t)|^2 \d t  \; (i=1,2,3,4)$ and $\mathcal{R}_2(x) $ is the next largest.
Now, from the various estimates we see that our $\mathcal{R}_1(x)$ term is given by
\begin{align*}
(1+o(1)) \frac{T}{x^2} \log^2 T & \qquad \hbox{ if } 1 \le x \le
(\log T)^{3/4} \;,\\
o (T\log T) & \qquad \hbox{ if } \quad (\log T)^{3/4} < x \le (\log T)^{3/2} \;, 
\end{align*}
and by
\begin{align*}
 \left(1 + o_K (1) \right) T \log x \left( 1 -2 \frac{\log x}\l +
2 \sum^K_{k=1} \frac{(k-1)!}{(2k)!} 
\left( \frac{\log x} \l \right)^{2k} \right)  
& \qquad \hbox{ if } \quad (\log T)^{3/2} < x \leq   T^{1-\varepsilon} \;.
\end{align*}
In each of these ranges, it happens that the $\mathcal{R}_2(x)$ term is 
$o(\mathcal{R}_1(x))$. Hence, 
taking $x = T^\alpha$ in \eqref{Sum over gammas}, we find that
for $0 < \alpha <1$ and $T$ large,
\begin{align*}
&   2\pi \, \sum_{0 < \g, \g'  \leq T}
T^{\i\alpha (\g  - \g')} w (\g  - \g' )\\
& = (1 + o(1))  T^{1-2\alpha}  \log^2 T 
+ (1 + o_K (1)) \alpha T \log T \left( 1 -  4\alpha  + 2 \sum^K_{k=1} \frac{(k-1)!}{(2k)!}
\left(  2 \alpha \right) ^{2k} \right)  +  o_K (1)\;.
\end{align*}

Using  \eqref{F_1} and noting that $F_1(\alpha, T)$ is an even function of 
$\alpha$, we have proved Theorem~\ref{thm:F1}.

\end{proof}

\section{Proof of Proposition \ref{prop:Akell}}\label{sec:arithmetic}

We prove Proposition \ref{prop:Akell}, which is the arithmetic
portion of the calculation.  We first reduce $A_{k,\ell}$ to
a sum involving the arithmetic functions~$\Lambda_j$.
In Section~\ref{sec:lemmas} we state some lemmas which are
needed in the calculation, and 
in Section~\ref{sec:Lambda_j} we evaluate the sums of the~$\Lambda_j$.
Then we complete the proof of Proposition \ref{prop:Akell}
in Section~\ref{sec:proof of Akell}

\subsection{Reduction of $A_{k,\ell}$}
Recall that 
\begin{equation}
A_{k,\ell}(x) = \sum_{n\le x} \alpha_k(n) \alpha_\ell(n),
\end{equation}
where $\alpha_k$ is given by~\eqref{eqn:alphak}.

\begin{lemma}\label{lem:Akell} We have
\begin{equation}
A_{k,0} (x) =
\begin{cases}
x \log x + O (x)  \quad& \text{ if }\ \  k=0 \,,\\
- x \log^2 x + O(x \log x) \quad& \text{ if }\ \  k=1 \,,\\
O\bigl(x^{\tfrac12+\varepsilon}\bigr)
\quad 
&\text{ if }\ \  k \geq 2 \;.
\end{cases}
\end{equation}
and
\begin{equation}
A_{k,1} (x) =
\begin{cases}
 x \log^3 x + O (x \log^2 x) \quad &
\text{ if }\ \ k=1\;,\\
O\bigl(x^{\tfrac12+\varepsilon}\bigr)
&
\text{ if }\ \ k \geq 2
\;. \end{cases}
\end{equation}
For $2\leq \ell \leq  k $ we have
\begin{align}\label{Aklinlemma}
A_{k,\ell} (x) & = (k-1) (\ell -1) \sum_{p \leq x } \sum_{q \leq \frac
xp} \log^3 p \, \log^3 q \left( \sum_{m \leq \frac x{pq}} \Lambda_{k-2}
(m) \Lambda_{\ell -2} (m) \right) \cr
& \quad+ \sum_{p \leq x} \log^4 p \left( \sum_{m \leq \frac xp} \Lambda_{k-1}
(m) \Lambda_{\ell -1} (m) \right) \cr
& \quad\quad + O (x \log^{k + \ell} x ) . 
\end{align}
\end{lemma}

\begin{proof}

The form of $\a_k (n)$ is different when $k=0$ and $1$ from what it 
is for larger $k$, and this will be reflected in our estimates for $ A_{k,\ell} (x)$.
We therefore treat these cases separately.

First consider the case of  $A_{k,0} (x)$. By the prime number theorem we have
\begin{equation}\label{A_{0,0}}
A_{0,0} (s)  = \sum_{n\leq x} \Lambda^2 (n) = x \log x + O (x) 
\end{equation}
and
\begin{equation}\label{A_{1,0}}
A_{1,0} (x)  = - \sum_{n\leq x} \Lambda^2 (n) \log n = - x \log^2 x + O
(x \log x) \,.
\end{equation}
For $k \geq 2$ we have
\begin{align}\label{eqn:Ak0estimate}
A_{k,0} (x) =\mathstrut& - \sum_{n\leq x} \Lambda (n) 
(\Lambda _{k-1} * \Lambda \log) (n) \cr
=\mathstrut&
 - \sum_{p^a \leq x} \log(p) (\Lambda_{k-1} * \Lambda \log)(p^a) \cr
\ll\mathstrut & 
\sum_{\ontop{p^a \le x}{a\ge k}} p^{a \varepsilon} \cr
\ll\mathstrut & x^{\tfrac12+\varepsilon} .
\end{align}
We have used the fact that $(\Lambda_{k-1} * \Lambda \log)(n) \ll n^\varepsilon$
and that this function vanishes unless $n$ is a product of at least $k$ (not necessarily distinct) primes 
and $k\ge 2$.

Next we consider $A_{k,1} (x)$ for $k \geq 1$.  By the prime number theorem,
\begin{equation}\label{A11}
 A_{1,1} (x) = \sum_{n\leq x} (\Lambda (n) \log n)^2 = x \log^3 x + O
(x\log^2 x) \;. 
\end{equation}
If $k \ge 2$ then
\begin{align}
 A_{k,1}(x) = \mathstrut & \sum_{n\leq x} \Lambda (n)   \log n \, 
\left(\Lambda_{k-1} * \Lambda \log \right) (n) \cr
\ll \mathstrut &  x^{\tfrac12+\varepsilon},
\end{align}
exactly as in~\eqref{eqn:Ak0estimate}.

We now come to the general case of $A_{k,\ell} (x)$ with $k \geq \ell \geq 2$. 
We have
\begin{align}\label{A_{k,ell}(x)}
A_{k,\ell}(x)
=\mathstrut & \sum_{n\leq x} \a_k(n) \a_l(n)  \cr
=\mathstrut &  \sum_{n\leq x} \left( \sum_{p^a|n} a
\log^2 p \;\Lambda_{k-1} \left( \frac{n}{p^ a}\right)\right) \left(
\sum_{q^b|n} b \log^2 q \;\Lambda_{\ell -1} \left( \frac n{q^b} \right)
\right) \cr
=\mathstrut &  \sum_{n\leq x} \left( \sum_{p|n} 
\log^2 p \;\Lambda_{k-1} \left( \frac{n}{p}\right)\right) \left(
\sum_{q|n}  \log^2 q \;\Lambda_{\ell -1} \left( \frac n{q} \right)
\right) \cr
&\quad + O(x\log^{k+\ell} x )\cr
=\mathstrut & \sum_{p \leq x}
    \sum_{\substack{q \leq \frac x{p} \\ q \ne p }}
   \log^2 p \log^2 q \left(
    \sum_{\substack {n \leq x\\ p q |n }}
     \Lambda_{k-1} \left( \frac n{p}\right)
     \Lambda_{\ell-1} \left( \frac n{q} \right) \right) \cr
& \quad +\sum_{p \leq x} \log^4 p \left(
\sum_{\substack{ n \leq x \\ p|n} }
     \Lambda_{k-1} \left( \frac n{p} \right)
      \Lambda_{\ell -1} \left( \frac n{p} \right)
\right) \cr
&\quad + O(x\log^{k+\ell} x )\cr
 =\mathstrut & {A}_1 + {A}_2 + O(x\log x^{k+\ell} x )\;,
\end{align}
say. The estimate on the third line above was done as follows.
By symmetry, it is sufficient to estimate
\begin{equation}\label{eqn:abtobound}
 \sum_{\substack{p^a \leq x\\ a\ge1}}
     \sum_{\substack{q^b \leq \frac x{p^a}\\ b\ge 2}}
     ab \log^2 p \; \log^2 q \left( \sum_{m \leq \frac x{p^a q^b}}
\Lambda_{k-1} (mq^b) \; \Lambda_{\ell -1} (mp^a)
\right) \;.
\end{equation}
By the trivial bound \eqref{Lambda_j bound} this is
\begin{align} \label{usetrivialbound}
\ll\mathstrut & \sum_{\substack{p^a \leq x\\ a\ge1}}
     \sum_{\substack{q^b \leq \frac x{p^a}\\ b\ge 2}}      ab \log^2 p \; \log^2 q \left( \sum_{m \leq \frac x{p^a q^b}}
(\log mq^b)^{k-1} \; (\log mp^a)^{\ell-1}
\right) \cr
\ll\mathstrut & x (\log x)^{k+\ell-2}
       \sum_{\substack{p^a \leq x\\ a\ge1}}
             \frac{a\log^2 p}{p^a} 
     \sum_{\substack{q^b \leq \frac x{p^a}\\ b\ge 2}}    
      \frac{b \log^2 q}{q^b}
\cr
\ll\mathstrut & x (\log x)^{k+\ell-2}
       \sum_{\substack{p^a \leq x\\ a\ge1}}
             \frac{a\log^2 p}{p^a}
\cr
\ll\mathstrut & x (\log x)^{k+\ell},
\end{align}
as claimed.  One can use Lemma~\ref{Lambdaj pn} on the inner summand
of \eqref{eqn:abtobound} to improve this bound by a power of
$\log x$, but this will not affect our final result.

It is clear that ${A}_2$ equals the second main term in
\eqref{Aklinlemma}, so it remains to put ${A}_1$ in the appropriate
form.  The point is that if $p$ is prime and $p\nmid m$ then
\begin{equation}\label{eqn:Lambda pm}
\Lambda_k(p m) = k (\log p) \Lambda_{k-1}(m).
\end{equation}
See the proof of Lemma~\ref{Lambdaj pn}.
Thus
\begin{align}\label{{A}_1} 
{A}_1 = \mathstrut &
 (k-1)(\ell-1)\sum_{p \leq x} \sum_{\substack{ q \leq \frac xp\\ q \ne p}} 
\log^3 p \,\log^3 q
\left(
\sum_{\substack{m \leq \frac x{pq} \\ (m,pq)=1} }
        \Lambda_{k-2} (m) \Lambda_{\ell -2} (m)
\right) 
\cr
&\mathstrut +
\sum_{p \leq x} \sum_{\substack{ q \leq \frac xp\\ q \ne p}}
\log^2 p \,\log^2 q \left( \sum_{\substack{ m \leq \frac x{pq}  \\ (m,qp)>1}} \Lambda_{k-1} (mq)
\Lambda_{\ell -1} (mp) \right) .
\end{align}
The second term can be estimated using the trivial bound~\eqref{Lambda_j bound},
exactly as in \eqref{usetrivialbound}, showing that it is
$\ll x (\log x)^{k+\ell}$.
For the first term, removing the conditions $(m, pq)=1$
and $q\not = p$ and estimating with the trivial bound~\eqref{Lambda_j bound}
gives an even smaller error term.  This completes the proof of
Lemma~\ref{lem:Akell}.

\end{proof}

%%%%%%%%%%%%%%
%%%%%%%%%%%%%%
%%%%%%%%%%%%%%
%%%%%%%%%%%%%%

\subsection{Some lemmas}\label{sec:lemmas}
The following lemmas concerning
the $\Lambda_j$ function
and
sums over primes
are required in the next subsection.

The first Lemma is an improvement on the trivial bound \eqref{Lambda_j bound}
for $\Lambda_j(n)$ when $n$ has a known prime factor.

\begin{lemma}\label{Lambdaj pn}
If $p|m$, then
$$ \Lambda_k (m) \leq k \log p (\log m)^{k-1}. $$
\end{lemma}

\begin{proof} If $(p,n)=1$ then
$$ \Lambda_k (p^a n) = \sum^{\min (a,k)}_{j=1} \binom{k}{j} \Lambda_j
(p^a ) \Lambda_{k-j} (n) = \sum^{k}_{j=1} \binom{k}{j} \binom{a-1}{j-
1} (\log p)^j \Lambda_{k-j} (n) \,.$$
Setting $a=1$, we obtain~\eqref{eqn:Lambda pm}.

Now
$$ \binom{a-1}{j-1} = \frac{(a-1) (a-2) \cdots ( a-j +1)}{(j-1)!} \leq
\frac{a^{j-1}}{(j-1)!} \;.
$$
Hence
\begin{align*}
\Lambda_k (p^a n) & \leq    \sum^{k}_{j=1}  \frac kj
\binom{k-1}{j-1} \frac{a^{j-1}}{(j-1)!} (\log p)^{j} (\log n)^{k-j} \\
& \leq k \log p \sum^{k}_{j=1} \binom{k-1}{j-1} (\log p^a )^{j-1}
(\log n)^{(k-1)-(j-1)} \;.
\end{align*}
Setting $j-1 =  i$, this is
\begin{align*}
& = k \log p \sum^{k-1}_{i=0} \binom{k-1}{i}
(\log p^a)^i (\log n)^{(k -1) -i} \\
& \leq k \log p (\log p^a n)^{k-1} \;.
\end{align*}
\end{proof}

\begin{lemma}\label{lem:logplogxp}
If $u\ge 2$ and $v \geq 1$ then
 \begin{equation}
  \sum_{p \leq x} \frac{\log^{u}p}p \left( \log \frac xp \right)^{v}
=  \frac{(u-1)! v!}{(u+v)!} \log^{u+v} x
+
O\left( \frac{(u-1)! v!}{(u+v-1)!} \log^{u+v-1} x \right) \;.
\end{equation}
\end{lemma}

\begin{proof}
Let $F(t)=\sum_{p\leq t }\log p/p$,  so  $F(t)=\log t +E(t)$, where $E(1)=0$ and $E(t) \ll 1$.
We have
\begin{align*}
  \sum_{p \leq x} \frac{\log^{u} p}p
  \left( \log \frac xp \right)^{v}
=\mathstrut &
\int_{1}^{x}\log^{u-1} t \, \left(\log\frac{x}{t}\right)^{v}\,dF(t) \\
=\mathstrut &\int_{1}^{x}\log^{u-1} t \, \left(\log\frac{x}{t}\right)^{v}\,d (\log t)
+\int_{1}^{x}\log^{u-1} t \, \left(\log\frac{x}{t}\right)^{v}\,dE(t) \\
=\mathstrut& I+J,
\end{align*}
say. In $I$ replace $t$ by $x^{\theta}$ and use Euler's Beta-integral to obtain
\begin{equation}
I= \log^{u+v} x \int_{0}^{1} \theta^{u-1} (1-\theta)^{v}\,d\theta
= \frac{(u-1)! v!}{(u+v)!}\log^{u+v} x,
\end{equation}
which is the main term above.
In the error term we integrate by parts and  find that
\begin{equation}
J= E(t) \log^{u-1} t \left(\log \frac{x}{t}\right)^{v}\bigg|_{1}^{x}
-\int_{1}^{x} E(t) \left(
(u-1) \log^{u-2} t \left(\log\frac{x}{t}\right)^{v}
-
v \log^{u-1} t  \left(\log \frac{x}{t}\right)^{v-1}
\right) \frac{dt}{t}.
\end{equation}
The first term vanishes.
In the integral make the change of variable
$t=x^{\theta}$ and again use the Beta-integral to obtain
\begin{align*}
J \ll \mathstrut & \log^{u+v-1}x\   (u-1) \int_{0}^{1} \theta^{u-2} (1-\theta)^{v}\,d\theta
+\log^{u+v-1}x\   v \int_{0}^{1} \theta^{u-1} (1-\theta)^{v-1}\,d\theta  \\
=\mathstrut &2 \frac{(u-1)! v!}{(u+v-1)!} \log^{u+v-1} x,
\end{align*}
as claimed.
\end{proof}

\subsection{The Estimation of \,$ \sum_{n \leq x} \Lambda_k (n) \Lambda_\ell(n)$}\label{sec:Lambda_j}

We evaluate the sums over $\Lambda_j$ which appear in Lemma~\ref{lem:Akell}

Let
$$
 \S_{k, l} (x) = \sum_{n \leq x} 
\Lambda_k (n) \Lambda_\ell(n) \,.
$$
In this subsection we prove the following theorem.
\begin{theorem}\label{thm:Sksum} 
If  $k > \ell \geq 1$, then
\begin{equation} \label{Skl}
\S_{k, \ell} (x) \ll   x \log^{k + \ell - 2} x \,.
\end{equation}
If $k \geq 1$, then
\begin{equation}\label{Skk} 
 \S_{k,k} (x) = \frac{k!}{(2k-1)!} x \log^{2k-1} x 
+ O(x \log^{2k-2}x)  \,.
\end{equation}
\end{theorem}

The theorem is proved by induction, using the following proposition.

\begin{proposition}\label{prop:Skl}
For $k\ge\ell\ge1$ we have
\begin{equation}\label{recursion}
 \S_{k, \ell} (x) 
= \ell \sum_{p \leq x} \log^2 p \ %
\S_{k-1,\ell-1}({x}/{p})
+ O(x \log^{k + \ell -2}x) .
\end{equation}
\end{proposition}

\begin{proof}

We  assume $k \geq \ell$ and begin by unfolding $\Lambda_k$ in the sum:
\begin{align}\label{Skl formula}
\S_{k, \ell} (x) & = \sum_{n\leq x} \left( \sum_{d|n} \Lambda(d)
\Lambda_{k-1} \left( \frac nd \right) \right) \Lambda_\ell (n) \\
& = \sum_{d\leq x} \Lambda (d) \left( \sum_{m \leq x/d} \Lambda
_{k-1} (m) \Lambda_\ell (md)\right)     \notag \\
& = \sum_{p^a \leq x} \log p \left( \sum_{m \leq
x/ p^a} \Lambda_{k-1} (m) \Lambda_\ell (mp^a) \right)  \notag     \\
& = \sum_{p \leq x} \log p \left( \sum_{m \leq x/p} \Lambda_{k-1}
(m) \Lambda_\ell (mp)\right)    \notag    \\
&  \qquad + \sum_{\substack{p^a \leq x \\ a \ge 2 }} \log p \left( \sum_{m \leq
x/ p^a} \Lambda_{k-1} (m) \Lambda_\ell (mp^a) \right)    \notag    \\
& = \Sigma_1 + \Sigma_2,   \notag
\end{align}
say. We split $\Sigma_1$ into two sums $\Sigma_{1, 1}$ and 
$\Sigma_{1, 2}$ according to whether $m$ in the inner sum is or is not
coprime to $p$. By \eqref{eqn:Lambda pm}, if $(m, p) =1$ then 
$\Lambda_\ell (mp) = \ell \Lambda_{\ell-1} (m) \log p$. Hence
\begin{align*}
\Sigma_{1, 1}    
 = \ell \sum_{p \leq x} \log^2 p \left( \sum_{\substack{ m \leq 
x/p\\ (m,p)=1 }} \Lambda_{k-1} (m) \Lambda_{\ell -1} (m) \right)\;.
\end{align*}
By Lemma~\ref{Lambdaj pn}, removing the coprimality condition here
introduces a change of 
$$
\ll  \ell \sum_{p \leq x} \log^2 p  \left(\frac{x}{p^2}  {k\ell \log^2 p} 
\log^{k+\ell-4}x \right)  \ll   x \log^{k + \ell -4} x \,.
$$
That is,
 \begin{align*}
\Sigma_{1, 1}    
 = \ell \sum_{p \leq x} \log^2 p \left( \sum_{\substack{ m \leq x/p  }} 
 \Lambda_{k-1} (m) \Lambda_{\ell -1} (m) \right)
 + O\left (x \log^{k + \ell -4} x \right)
 \;.
\end{align*}
 
For $\Sigma_{1, 2}$ we find that
$$ 
 \Sigma_{1, 2} =   \sum_{p \leq x} \log p 
 \left( \sum_{\substack{ m \leq x/p \\ p|m}} \Lambda_{k-1}
(m) \Lambda_l (mp) \right) \;.
 $$
By Lemma~\ref{Lambdaj pn} this is 
\begin{align*}
& \ll     \sum_{p \leq x} \log p \sum_{r \leq x/p^2} \Lambda_{k-
1} (rp) \Lambda_ \ell (r p^2)  
  \ll k \ell \sum_{p \leq x} \log^3 p \sum_{r \leq x/p^2 } (\log x)^{k +  \ell -3} \\
& \ll  x (\log x)^{k + \ell -3} \sum_{p \leq x} \frac{\log^3 p}{p^2}  
  \ll  x (\log x)^{k + \ell -3}\;.
\end{align*}
Hence, combining   $\Sigma_{1, 1}$ and $\Sigma_{1, 2}$, we obtain 
\begin{align}\label{Sigma 1}
\Sigma_{1}    
 = \ell \sum_{p \leq x} \log^2 p \left( \sum_{\substack{ m \leq x/p  }} 
 \Lambda_{k-1} (m) \Lambda_{\ell -1} (m) \right)
 + O\left (\ell x (\log x)^{k + \ell -2}\right)
 \;,
\end{align} 
which is the main term in the Proposition.

By 
the trivial bound \eqref{Lambda_j bound}  and  Lemma~\ref{Lambdaj pn} we have 
 \begin{align*}
\Sigma_{2} & \ll \ell \sum_{\substack{p^a \leq x\\ a \geq 2}} \log^2 p \left(
\sum_{m \leq x/p^a} (\log x)^{k + \ell -2} \right) \\
& \ll \ell x (\log x)^{k + \ell -2} \sum_{\substack{p^a \leq x\\ a \geq
2 }} \frac{\log^2 p}{p^a} \\
& \ll x (\log x)^{k + \ell -2} \;.
\end{align*}
Combining this with \eqref{Sigma 1} completes the proof of
Proposition~\ref{prop:Skl}.
\end{proof}
 
\begin{proof}[Proof of Theorem \ref{thm:Sksum}]
We first prove the bound~\eqref{Skl}.
If $k \ge 2$ then using the basic properties of
$\Lambda$ and $\Lambda_k$ we have
\begin{align*} \S_{k,1} (x) 
 =\mathstrut &
\sum_{n \leq x} \Lambda_k (n) \Lambda(n) \cr
=\mathstrut &
\sum_{p^a \leq x} \Lambda_k (p^a) \log p   \cr
 \le \mathstrut & \sum_{\substack{p^a \leq x \\ a\ge k } } (\log p^a)^{k+1}
\cr
 \ll \mathstrut &  x^{\frac1k} \log^{k+1} x \cr
\ll \mathstrut &  x^{\frac 12 + \varepsilon} ,
\end{align*}
which is much smaller than the claimed bound.

Now suppose \eqref{Skl} holds for some $\ell >1$ and all $k>\ell$.
By Proposition~\ref{prop:Skl}, the induction hypothesis,
and Lemma~\ref{lem:logplogxp}:
\begin{align*}
\S_{k , \ell+1}(x) & = (\ell+1) \sum_{p \leq x} \log^2 p
\;\S_{k-1, \ell} \left( x/p \right) 
+O(x (\log x)^{k + \ell -1} )\\
& \ll  \sum_{p \leq x} \log^2 p 
\ \frac{x}{p}\left( \log  \frac{x}{p} \right)^{k +  \ell -3}
  + O\left(x (\log x) ^{k + \ell -1}\right) \\
&\mathstrut \ll  x (\log x)^{k + \ell -1}.
\end{align*}
as required.
This proves \eqref{Skl}.

Now we prove~\eqref{Skk}.  When $k=1$ we have
\begin{equation}
\S_{1,1} (x) = \sum_{n\leq x} \Lambda^2 (n) = x\log x + O(x)\;,
\end{equation}  
so \eqref{Skk} holds in this case. Suppose \eqref{Skk} holds
for some $k>1$.  Then by Proposition~\ref{prop:Skl}, the
induction hypothesis, and Lemma~\ref{lem:logplogxp}:
\begin{align}
\S_{k+1,k+1 } (x) =\mathstrut&
(k+1) \sum_{p\le x} \log^2 p\ \S_{k,k}\left(\frac{x}{p} \right) 
    + O(x \log^{2k}x ) \cr
= \mathstrut &\frac{(k+1)!}{(2k-
1)!} x \sum_{p\leq x} \frac{\log^2p}{p} \left( \log \frac xp
\right)^{2k-1}   \\
& + O\left(x \sum_{p\leq x} \frac{\log^2 p}p \left( \log \frac
xp\right)^{2k-2}\right) + O(x \log^{2k}x )  \\
=\mathstrut & \frac{(k+1)!}{(2k+1)!} x \log^{2k+1} x  
 + O\left( x \log^{2k}x \right),
\end{align}
as required.
\end{proof}
%%%%%%%%%%%%%%%%%%%
%%%%%%%%%%%%%

\subsection{Completion of the proof of Proposition \ref{prop:Akell}}\label{sec:proof of Akell}
We are now ready to estimate $A_{k,\ell} (x)$. 

\begin{proof}
Suppose first that $k>\ell\ge 3 $.  
Using Theorem~\ref{thm:Sksum} and
Lemma~\ref{lem:Akell} we have
\begin{align}\label{FinForm Akl}
A_{k,\ell} (x) \ll\mathstrut & 
 x \sum_{p\leq x} \sum_{q \leq \frac xp}
    \frac{\log^3 p \, \log^3 q}{pq} \left( \log \frac{x}{pq}
\right)^{k + \ell -6} \cr
& \mathstrut +  x \sum_{p \leq x} \frac{\log^4 p}p \left( \log \frac
xp \right)^{k+ \ell -4} +  O\left( x (\log x)^{k+\ell} \right)\;.
\end{align}
By Lemma~\ref{lem:logplogxp} the first term is
\begin{equation}
  \ll x\sum_{p\le x} \frac{\log^3 p}{p} \left( \log \frac xp
\right)^{k+\ell -3}
\ll x\log^{k+\ell}x ,
\end{equation}
and the second term is also $\ll x \log^{k+\ell} x$.

When $k >\ell$ and $\ell =2$, we obtain the same bound.  The only
difference is that this time the first term in \eqref{FinForm Akl}
is omitted 
because $\sum_{n \leq x} \Lambda_{k-2} (n) \Lambda_0
(n) = \Lambda_{k-2} (1) =0$ when  $k >2$.
This proves \eqref{eqn:Aklbnd}.

Next suppose that $k\ge 3$.  By 
Lemma~\ref{lem:Akell} and then
Theorem~\ref{thm:Sksum} we have
\begin{align}
A_{k,k} (x)  =\mathstrut & (k-1)^2
\sum_{p \leq x} \sum_{q \leq  x/p}
\log^3 p \log^3 q
\sum_{m\le \frac{x}{pq}} \Lambda_{k-2}(m)^2 \cr
 &\mathstrut +\sum_{p \leq x} \log^4 p 
\sum_{m\le \frac{x}{p}} \Lambda_{k-1}(m)^2 \label{eqn:Akkmainterm}
+ O\left(x \log^{2k} x \right) \; \\
=\mathstrut & \frac{(k-1)(k-1)!}{(2k-5)!} 
\sum_{p \le x}
\frac{\log^3 p}{p}
\sum_{q \le  x/p}
\frac{\log^3 q}{q} 
\left(
\left(\log\frac{x}{pq}\right)^{2k-5} 
+ O\left(\frac{x}{pq} \log^{2k-6} x \right)
\right)\cr
 &\mathstrut +\frac{(k-1)!}{(2k-3)!} \sum_{p \leq x} \log^4 p 
\left(
\left(\log\frac{x}{pq}\right)^{2k-3} 
+ O\left(\frac{x}{p}\log^{2k-4} x\right)
\right)
+ O\left(x \log^{2k} x \right).
\end{align}
Applying Lemma \ref{lem:logplogxp} we obtain
\begin{align}
A_{k,k}(x)=\mathstrut &
\frac{2 (k-1)(k-1)!}{(2k-2)!} x \sum_{p\le x}
\frac{\log^3 p}{p}\left(\log \frac{x}{p}\right)^{2k-2} 
\cr
&\mathstrut +  \frac{6 (k-1)!}{(2k+1)!} x\log^{2k+1} x 
+ O\left(x \log^{2k} x \right) \cr
=\mathstrut & \frac{4 (k-1)(k-1)!}{(2k+1)!} x\log^{2k+1} x
 + \frac{6 (k-1)!}{(2k+1)!} x\log^{2k+1} x
+ O\left(x \log^{2k} x \right) \cr
=\mathstrut & 2\frac{(k-1)!}{(2k)!} x\log^{2k+1} x
+ O\left(x \log^{2k} x \right),
\end{align}
as claimed.

It remains to do the case $k =2$.  The
only change in the above analysis is in the first term of 
\eqref{eqn:Akkmainterm}.  Since $\Lambda_0(m)=1$ if $m=1$
and 0 otherwise, using Lemma~\ref{lem:Akell} and then
Theorem~\ref{thm:Sksum}  and Lemma~\ref{lem:logplogxp} we have
\begin{align*}
A_{2,2}(x)
=\mathstrut & \sum_{p \leq x} \log^3 p \sum_{q \leq \frac
xp} \log^3 q 
+\sum_{p\le x} \log^4 p \sum_{m\le x/p} \Lambda(m)^2 
+O(x \log^4 x) \\
& = x \sum_{p \leq x} \frac{\log^3 p}p \left( \log^2 \frac xp + O \left(
\log x \right) \right) 
+x\sum_{p\le x} \log^4 p \log\frac{x}{p}
+O(x \log^4 x) \\
& = \frac{2! \,2!}{5!} x  \log^5 x + \frac{3!1!}{5!} x \log^5 x+O (x \log^4 x) \\
&= \frac 1{12} x \log^5 x + O (x \log^4 x)\;.
\end{align*}
Note that this is the same as the general case with $k=2$.
This completes the proof.

Recalling equation \eqref{A(x) 3},
\begin{equation}
\A (x) = \sum^K _{k=0} A_{k,k} (x) \l^{-2k} + 2 \sum^{K-1}_{\ell =0}
\sum_{\ell < k \leq K} A_{k,\ell} (x) \l^{-(k+ \ell)}\;,
\end{equation}
and using the fact that $x/\l \ll 1$ if $x=T^\alpha$ with $0<\alpha<C_0$,
gives the final statement in Proposition~\ref{prop:Akell}.
\end{proof}

%%%%%%%%%%%%%%%%%%%%%%%%%%%%%%%%%%%%%%%
%%%%%%%%%%%%%%%%%%%%%%%%%%%%%%%%%%%%%%%
%%%%%%%%%%%%%%%%%%%%%%%%%%%%%%%%%%%%%%%
%%%%%%%%%%%%%%%%%%%%%%%%%%%%%%%%%%%%%%%

\section{Proof of the Explicit Formula}\label{sec:explicitformula}
We prove Lemma~\ref{lem:logderiv} and Proposition~\ref{prop:explicitformula}.

\subsection{Proof of Lemma~\ref{lem:logderiv}}
\begin{proof} % [Proof of Lemma~\ref{lem:logderiv}]

Since $\Gamma(s)$ has simple poles at $s = 0, -2,-4,
\ldots,$  $L (s)$ has simple poles  with residue $-1$ at $s= -2, -4, \ldots$
and a  simple pole with residue 1 at $s =1$.  It is not difficult to show that 
$L(s)$ has only simple real zeros at  $m_1 \approx 7.6, m_2 \approx 2.8, 
m_3 \approx - 2.6, \dots$, with $m_j \to -\infty$ as $j \to\infty$.  
Multiplying both sides of \eqref{xi'/xi} by $\xi (s)$ and calculating the
logarithmic derivative, we obtain
\begin{equation}\label{xi''/xi'}%
\frac{\xi''}{\xi'} (s) = \frac{\xi'}\xi (s) + \frac{L' (s) + (\zeta'
/\zeta )' (s)}{L(s) + \zeta' /\zeta (s)} \;.
\end{equation}
Suppose now that $ \varepsilon > 0$. Then there exists an absolute constant
$C_1$ such that for $\sigma \geq 1 + \varepsilon$,
\begin{equation}\label{zeta ineq 1}
\left| \frac{\zeta'}\zeta (s) \right| \leq \frac{\zeta'}{\zeta} (\sigma)
\leq C_1 \varepsilon^{-1}
\end{equation}
and  
\begin{equation}\label{zeta ineq 2}
\left| \left( \frac{\zeta'}{\zeta} \right)' (s) \right| \leq
 C_1 \varepsilon^{-2}\;. 
\end{equation}
Hence, by \eqref{L defn}, there exists an absolute constant $C_2$ such that
\begin{equation}\label{zeta ineq 3} 
\left| \frac{\zeta'}\zeta (s) L(s)^{-1} \right| \leq \frac{4 C_1}{\varepsilon
\log (|s|+2)} < \frac 12 
\end{equation}
for $\sigma \geq 1 + \varepsilon$ and $|t| \geq T_\varepsilon = C_2
e^{8C_1/\varepsilon}$.  Now let $K$ be an arbitrary large integer.  Then
by \eqref{L ineq 2}, \eqref{xi''/xi'},\eqref{zeta ineq 1}, and \eqref{zeta ineq 3},
\begin{equation}\label{xi'' / xi' 2}
\frac{\xi''}{\xi'} (s) = L(s) + \frac{\zeta'}\zeta (s) + \left(
\frac{\zeta'}\zeta \right)' (s) L (s)^{-1} \sum^{K-1}_{j=0} \left( -
\frac{\zeta'}\zeta (s) L(s) ^{-1} \right)^j + 
O\left(\frac{1}{\varepsilon^2 2^K}\right)    
\end{equation}
for $\sigma \geq 1 + \varepsilon $ and 
$|t| \geq T_\varepsilon = C_2 e^{8C_1/\varepsilon}$.  
Using the definition of the $j$-fold von~Mangoldt function~\eqref{eqn:vonM},
this can be rewritten as
\begin{equation}
\frac{\xi''}{\xi'} (s) = L(s) + \sum^\infty_{n=1} n^{-s} \left( -
\Lambda (n) + \sum^{K-1}_{j=0} \frac{(\Lambda_j * \Lambda
\log)(n)}{L(s)^{j+1}} \right) + O\left(\frac{1}{\varepsilon^2 2^K}\right)\,.  
\end{equation}
This completes the proof of Lemma~\ref{lem:logderiv}.

\end{proof}

\subsection{Comparison to the Hardy $Z$-function}\label{sec:Z}

We now indicate why
the pair correlation functions
for the zeros of $\xi'$ and $Z'$ are equal to leading order.

Suppose $\mathcal X(s)=\Upsilon(s) \zeta(s)$ and let
$\mathcal L(s) = \displaystyle{\frac{\Upsilon'}{\Upsilon}(s)}$.
We have
\begin{equation}
\frac{\mathcal X''}{\mathcal X'} (s) = {\mathcal L}(s) + \frac{\zeta'}\zeta (s) +
\frac{{\mathcal L}' (s) + (\zeta'
/\zeta )' (s)}{{\mathcal L}(s) + \zeta' /\zeta (s)} \;.
\end{equation}
By choosing $\Upsilon$ appropriately, one can obtain either the
Riemann $\xi$-function or the Hardy $Z$-function.  In either case,
\begin{equation}
\mathcal L(s) \sim \frac12 \log s 
\end{equation}
and
\begin{equation}
\mathcal L'(s) \ll s^{-1},
\end{equation}
which is all that was used in the calculation of the form factor
$F_1(\alpha; T)$.  

The lower order terms in $\mathcal L(s)$ are different in those two cases,
and this should have an effect on the lower order terms of
$F_1(\alpha;T)$.  Presumably the lower order terms also have an arithmetic
component, so both should differ from the analogous expression from 
random matrix theory.

\subsection{Proof of Proposition~\ref{prop:explicitformula}}

\begin{proof} % [Proof of Proposition~\ref{prop:explicitformula}]

We will integrate both sides of \eqref{Dir series} against the
following kernel:
\begin{equation}\label{k defn} 
k(w,s) = \frac{2\sigma-1}{(w - (s-1/2)) (w - (1/2 - \bar{s}))}\;.
\end{equation}
It is easy to see that
\begin{equation}
k (w, 1-\bar{s} ) = - k (w,s)
\end{equation}
and
\begin{equation} k (-w,s) = k (w, \bar{s})\;.
\end{equation}
Moreover, as a function of $w$, $k (w,s)$ has simple poles at $w=s-1/2$
and $w=1/2 - \bar{s}$ with residues 1 and $-1$, respectively.  From the
partial fraction decomposition for $\zeta'/\zeta (w)$ and the fact that
$\zeta (w)$ has $O (\log T)$ zeros with ordinates in the interval $[T,
T+1]$, it follows that one can find an increasing, unbounded sequence
$\{ T_j \}^\infty_{j=1}$ such that
\begin{equation}\label{bd 1} %
 \frac{\zeta'}{\zeta} (u + \i T_j) \ll \log^2 T_j
\end{equation}
and 
\begin{equation}\label{bd 2}%
 \left( \frac{\zeta'}\zeta \right)' (u + \i T_j) \ll \log^3 T_j
\end{equation}
uniformly for $-1 \leq u \leq 2$.  Using these in 
\eqref{xi'/xi}  and \eqref{xi''/xi'}, we find
that
\begin{equation}\label{bd 3}% 
\frac{\xi''}{\xi'} (u+ \i T_j) \ll \log^2 T_j 
\end{equation}
uniformly for $-1 \leq u \leq 2$.  We now write 
\begin{equation}\label{I_j 1}% 
I_j = \frac 1{2\pi_\i} \int_{\R_j} \frac{\xi''}{\xi'} (w + 1/2) k (w,s)
x^w \d w \;,
\end{equation}
where $x \geq 1,  5/4 < \sigma <2$, and $\R_j$ is the positively oriented
rectangle with vertices at $c \pm \i T_j , -U \pm \i T_j$, where $c = 
1/2 + \varepsilon$, $\varepsilon < 1/8$, and $U$ is a large positive
number.  The integrand has simple poles at $w = \i \gamma$, $w = s -
1/2$, and $w = 1/2 - \bar{s}$.  Now, since $\varepsilon < 1/8$, $\Re\, (s -
1/2) = \sigma - 1/2 > 3/4 > c$. Therefore $s - 1/2$ lies outside $\R_j$.  Thus, by
the calculus of residues,
\begin{equation}\label{I_j 2}%
 I_j = - \frac{\xi''}{\xi'} (1-\bar{s} ) x^{1/2 - \bar{s}} +
\sum_{|\gamma | \leq T_j} k (\i \gamma, s) x^{\i \gamma}\;.
\end{equation}
 
We now estimate the contributions of the horizontal and left edges of
$\R_j$ to $I_j$.   Besides  \eqref{bd 3}  we need  the estimate
\begin{equation}\label{bd 4}%
  \frac{\xi''}{\xi' } (w) \ll \log 2 |w| 
\end{equation}
for $\Re \, w < - 1/2$.  By   \eqref{xi'/xi},  \eqref{L defn}, \eqref{L ineq 2}, 
  and \eqref{xi''/xi'} this holds in $\Re \, w > 3/2$. Hence it holds in $Re\,w < - 1/2$ by 
the functional equation \eqref{fnc equ}. 
By \eqref{bd 3} and \eqref{bd 4}, the top and bottom edges
of $\R_j$ contribute
\begin{align*}
\ll & \int^{-1}_{-U} \frac{\log 2 |u+i T_j|}{(T_j -t)^2} x^u \d u +
\int^c_{-1} \frac{\log^2 T_j}{(T_j - t)^2} x^u \d u \\
\ll & x^{1/2 + \varepsilon} \; \frac{\log^2 T_j}{(T_j - t)^2} 
\end{align*}
to $I_j$.  The left edge contributes
\begin{align*}
\ll & \int^{T_j}_0 \; \frac{\log \sqrt{U^2 + v^2}}{U^2 + v^2} x ^{-U}\d v \\
\ll & \frac{x^{-U}}U \int^\infty_0 \frac{\log U + \log (1 + x^2)}{1 + x^2} \; \d x\\
\ll & x^{-U} \frac{\log U}U \;.
\end{align*}
Letting $U$ and $T_j$ both tend to infinity, we obtain
\begin{equation}\label{sum over zeros formula 1}
\sum_{\gamma} k (\i \gamma, s) x^{\i \gamma} = \frac{\xi''}{\xi'}
(1- \bar{s}) x^{1/2 - \bar{s}} + \frac 1{2\pi \i} \int^{c + \i \infty}_{c-
\i \infty} \frac{\xi''}{\xi'} (w + \frac 12 ) k (w,s) x^w \d w \;.
\end{equation}
 
We evaluate the integral here by replacing $\xi''/\xi' (w +
1/2)$  by \eqref{Dir series}.  Since this representation 
holds only when $|\Im\,w | \geq T_\varepsilon$, we have  
\begin{equation}
\begin{aligned}
 \frac 1{2\pi \i} \int^{c+\i \infty}_{c-\i \infty}  \frac{\xi''}{\xi'} &(w +
1/2) k (w,s) x^w \d w \\
=& \frac 1{2\pi \i} \int^{c + \i \infty}_{c-\i \infty} \left( L (w + 1/2) +
\sum^\infty_{n=1} \frac{a_K (n, w + 1/2)}{n^{w + 1/2}} \right) k (w,s)
x^w \d w\\
&+ \E_1 + \E_2 + \E_3 \,,
\end{aligned}
\end{equation}
where
\begin{align*}
\E_1 &=  \frac 1{2\pi \i} \int^{c+\i T_\varepsilon}_{c-\i T_\varepsilon}  \frac{\xi''}{\xi'} (w +
1/2) k (w,s) x^w \d w \,, \\
\E_2 &= \frac 1{2\pi \i} \int^{c + \i T_\varepsilon}_{c-\i T_\varepsilon} \left( L (w + 1/2) +
\sum^\infty_{n=1} \frac{a_K (n, w + 1/2)}{n^{w + 1/2}} \right) k (w,s)
x^w \d w \,,   \\
\E_3 &=  O \left( \int^{\infty}_{T_\varepsilon} \frac{1}{\varepsilon 2^K} \frac{x^{1/2
+ \varepsilon}}{ 1 + (v-t)^2} \d v \right)  \;.
\end{align*}
Now
\begin{align*}
\E_1 & \ll  \log T_\varepsilon \left( \int_{0}^{T_\varepsilon}  
\frac{x^{1/2 + \varepsilon}}{ 1 + (v-t)^2} \d v \right) \\
&\ll_\varepsilon  x^{1/2 + \varepsilon} \left( \frac 1{|t| +2} 
+ \frac 1{|t - T_\varepsilon| + 2} \right) \\
& \ll_\varepsilon \frac{x^{1/2+ \varepsilon}}{|t| + 2}      \;.
 \end{align*}
Furthermore, by \eqref{a_K (n,s) bd}
\begin{align*}  
\E_2 \ll &  x^{1/2 + \varepsilon} \int^{T_\varepsilon}_0 \left( \log (v+2) +
C_3 ^{K+2} \sum^\infty_{n=1} \frac{(\log n)^{K+2}}{n^{1 + \varepsilon}}
\right) \, \frac{dv}{(1 + |v-t|)^2} \\
& \ll x^{1/2 + \varepsilon} \left( \log T_\varepsilon + \left(
\frac{C_3}\varepsilon \right)^{K+3} \right) \left( \frac 1{|t|+2} +
\frac 1{|t - T_\varepsilon | +2} \right) \\
& \ll_{\varepsilon, K} \; x^{1/2 + \varepsilon} \left( \frac 1{|t| +2} +
\frac 1{|t-T_\varepsilon |+2} \right) \\
& \ll_{ \varepsilon, K} \; \frac{x^{1/2 + \varepsilon}}{ (|t| +2)} \; .
\end{align*} 
Clearly we also have 
$$
\E_3 \ll_{ \varepsilon, K} \; \frac{x^{1/2 + \varepsilon}}{ (|t| +2)} \; .
$$
Therefore,
\begin{equation}\label{Int form 1}
\begin{aligned}
 \frac 1{2\pi \i} \int^{c+\i \infty}_{c-\i \infty}  \frac{\xi''}{\xi'} &(w +
1/2) k (w,s) x^w \d w \\
=& \frac 1{2\pi \i} \int^{c + \i \infty}_{c-\i \infty} \left( L (w + 1/2) +
\sum^\infty_{n=1} \frac{a_K (n, w + 1/2)}{n^{w + 1/2}} \right) k (w,s)
x^w \d w\\
&+ O_{\varepsilon, K }\left( \frac{x^{1/2 + \varepsilon}}{|t| +2} \right) \;. 
\end{aligned}
\end{equation}
We split the integral on the right-hand side into two
parts, namely,
\begin{align}\label{I 1 + I 2}
\I_1  + \I_2   = & \frac{1}{2\pi \i} \int^{c + \i \infty}_{c-\i\infty} \left( L (w + 1/2) +
\sum_{n \leq x} \frac{a_K (n,w + 1/2)}{n^{w + 1/2}} \right) k(w,s) x^w \d w\\
+ & \; \frac 1{2\pi \i} \int^{c + \i\infty}_{c- \i \infty} \left( \sum_{n>x}
\frac{a_K (n, w + 1/2)}{n^{w + 1/2}} \right) k (w,s) x^w \d w \;.\notag
\end{align}
In $\I_1$ we pull the contour left to $-\infty$
and in doing so we pass a pole of $k (w,s)$ at $w = 1/2 - \bar{s}$,
the poles of $L (w + 1/2)$ at $w = 1/2, - 5/2, -9/2, - 13/2, \dotsc$, and
the poles of $1/L (w + 1/2)$ at the points $w = m_3 - 1/2, m_4 - 1/2, \dots$,
where the $m_j$'s are the zeros of $L(s)$.
We find that
\begin{equation}
\begin{aligned}
\I_1 & = - \left( L ( - \bar{s} ) + \sum_{n\leq x} \frac{a_K
(n, 1 - \bar s)}{n^{1-\bar s}} \right) x^{1/2 - \bar s}
- \sum^\infty_{m=1} k (-2m + 1/2, s) x^{-2m +1/2}\\
& \quad + k (1/2, s) x^{1/2}  
  + \sum_{n\leq x} \sum^K_{k=0} \frac{(\Lambda_k * \Lambda \log)
(n)}{\sqrt{n}} \sum^\infty_{j=3} \left( \frac xn \right)^{m_j - 1/2}
P_{k, j}  \left( \log \frac xn, s \right)\;, 
\end{aligned}
\end{equation}
where
\begin{equation}
\begin{aligned}
x^{m_j - 1/2} P_{k, j} (\log x,s) & ={\rm{Res}}_{w=m_j - 1/2} \, \frac{k
(w,s) x^w}{L (w + 1/2)^{k+1}} \\
& \ll \frac{(\log x)^k}{k!} x^{m_j - 1/2} | k (m_j - 1/2, s)| \\
& \ll \frac{(\log x)^k}{k!} x^{m_j - 1/2} |m_j -s|^{-2}\;.
\end{aligned}
\end{equation}
Using this and  \eqref{Alpha_k bound}, we find that the sum over
$m_j$ is
\begin{align*}
\ll & \sum_{n\leq x} \frac{(\log n)^{K+2}}{\sqrt{n}} \sum^K_{k=0}
\frac{(\log x/n)^k}{k!} \sum^\infty_{j=3} \left( \frac xn \right)^{m_j -
1/2} \frac 1{|m_j - s|^2 } \\
\ll & \sum_{n \leq x} \frac{(\log n)^{K+2}}{\sqrt{n}} \sum^K_{k=0}
\frac{ (\log x/n)^k}{k!} \left( \frac xn \right)^{m_3 - 1/2} \frac1{|s|}\\
\ll & (\log x)^{2K+2} |s|^{-1} x^{m_3 - 1/2} \sum_{n\leq x} \frac
1{n^{m_3}}\\
\ll & x^{1/2} (\log x)^{2K +2} |s|^{-1} \;.
\end{align*}
%%%%%%%
The sum over $m$  is
$$ 
\ll \sum^\infty_{m=1} \, \frac{x^{-2m +1/2}}{|2m-s|^2} \ll x^{-3/2}
 |s|^{-1} \;.
$$
Hence
\begin{equation}\label{I 1}
\begin{aligned}
\I_1  &= -  \left( L (1-\bar s) + \sum_{n\leq x} \frac{a_K (n, 1 - \bar
s)}{n^{1-\bar s}} \right) x^{1/2 - \bar s} + k ( 1/2, s) x^{1/2}  \\
 & \;+  O \left(x^{1/2} (\log x)^{2K +2}|s|^{-1} \right) \,.
\end{aligned}
\end{equation}
 
We treat $\I_2$ by moving the contour away to $+\infty$.
This time we pass the pole from $k (w, s)$ at $w = s - 1/2$ and the poles of
$ 1/L(w + 1/2)$ at $w = m_1 - 1/2$ and $m_2 - 1/2$. Thus we find that 
$$
\I_2 = - \left( \sum_{n>x} \frac{a_K (n,s)}{n^s}\right) x^{s - 1/2} -
\sum_{n>x} \sum^K_{k=0} \frac{(\Lambda_k * \Lambda \log) (n)}{\sqrt{n}}
\sum^2_{j=1} \left( \frac xn \right)^{m_j - 1/2} P_{k, j} \left( \log
\frac xn, s \right) \,.
$$
We treat the second of these two terms  as we did the corresponding term in $\I_1$ 
and find that it is
\begin{align*}
\ll & \frac 1{|s|^2} \sum_{n>x} \frac{(\log n)^{K+2} }{\sqrt{n} } \sum^K_{k=0}
\frac{(\log n/x)^k}{k!} \left( \left( \frac xn \right)^{m_1 -1/2} +
\left( \frac xn \right)^{m_2 - 1/2} \right)  \\
\ll & \frac 1{|s|^2} \sum_{n>x} \frac{ (\log n)^{2K+2} }{ \sqrt{n} } \left(
\left( \frac xn \right)^{m_1 - 1/2} + \left( \frac xn \right) ^{m_2 -
1/2} \right) \\
\ll & x^{1/2} (\log x)^{2K +2} |s|^{-2} \;.
\end{align*}
Hence,
\begin{equation}\label{I 2}
\I_2 = -x^{ s- \frac12}  \left( \sum_{n>x} \frac{a_K (n,s)}{n^s} \right) 
+ O \left( x^{1/2} (\log x)^{2K+2}|s|^{-2} \right)\;.
\end{equation}

Combining   \eqref{sum over zeros formula 1},  \eqref{Int form 1}, 
\eqref{I 1 + I 2}, \eqref{I 1}, and \eqref{I 2}, we see that 
\begin{align*}
\sum_{\gamma} k (\i \gamma, s) x^{\i \gamma} & = \frac{\xi''}{\xi'}
(1- \bar s) x^{1/2 - \bar s} - L (1-\bar s) x^{1/2 - \bar s}\\
& - x^{-1/2} \left( \sum_{n\leq x} a_K (n, 1-\bar s) \left( \frac xn
\right)^{1-\bar s} + \sum_{n>x} a_K (n,s) \left( \frac xn \right)^s
\right) \\
& + k (1/2, s) x^{1/2} + O_{\varepsilon, K} \left( x^{1/2 +
\varepsilon} |s|^{-1} \right) + O \left( x^{1/2} (\log x)^{2K+2}|s|^{-1} \right) \;. 
\end{align*}
By the functional equation  and \eqref{xi''/xi'},
$\xi''/\xi'(1-\bar s) = -\xi''/\xi'(\bar s) = -\xi'/\xi(\bar s) + O (1)$
for $5/4 < \sigma <2$. By \eqref{L defn} and \eqref{xi'/xi}
this equals $-1/2 \log \bar s +O(1)$. 
Hence, writing $\tau= |t|+2$, and using the definition \eqref{k defn}, 
we finally  obtain the explicit formula in Proposition~\ref{prop:explicitformula}.
\end{proof}

\end{document}